\documentclass[final,3p,times]{elsarticle}

\usepackage{graphicx,color}
\usepackage{amsmath,amssymb,amsfonts,amsthm}
\usepackage{mathdots,comment}
\usepackage{rotating}
\usepackage{todonotes}

\def\tablenotes{\bgroup\parfillskip=0pt plus 1fil
\leftskip=0pt\relax \rightskip=0pt
\vskip2pt\footnotesize}
\def\endtablenotes{\vskip1pt\egroup}

\def\sphline{\noalign{\vskip3pt}\hline\noalign{\vskip3pt}}

\def\Xint#1{\mathchoice
   {\XXint\displaystyle\textstyle{#1}}%
   {\XXint\textstyle\scriptstyle{#1}}%
   {\XXint\scriptstyle\scriptscriptstyle{#1}}%
   {\XXint\scriptscriptstyle\scriptscriptstyle{#1}}%
   \!\int}
\def\XXint#1#2#3{{\setbox0=\hbox{$#1{#2#3}{\int}$}
     \vcenter{\hbox{$#2#3$}}\kern-.5\wd0}}
\def\ddashint{\Xint=}
\def\dashint{\Xint-}
\def\xint{\Xint{\times}}

\newtheorem{theorem}{Theorem}
\newtheorem{lemma}[theorem]{Lemma}

\newdefinition{definition}[theorem]{Definition}
\theoremstyle{definition}

\newtheorem*{remarks}{Remarks}

\def\dkf{{\rm\,d}}
\def\I{{\rm i}}

\def\pr(#1){\left({#1}\right)}
\def\br[#1]{\left[{#1}\right]}
\def\fbr[#1]{\!\left[{#1}\right]\!}
\def\abs#1{\left|{#1}\right|}
\def\Jin{J_+^{-1}}
\def\CC{{\cal C}}
\def\half{\dfrac{1}{2}}
\def\E{{\rm e}}
\def\addtab#1={#1\;&=}

\newcommand{\sotodo}{\todo[color=green]}
\newcommand{\sotodoinline}{\todo[color=green,inline=true]}

\journal{Journal of Computational Physics}

\begin{document}

\begin{frontmatter}



\title{A fast and well-conditioned spectral method for singular integral equations}


\author[RMS]{Richard Mikael Slevinsky\corref{cor1}}
\ead{Richard.Slevinsky@umanitoba.ca}

\author[SO]{Sheehan Olver}
\ead{Sheehan.Olver@sydney.edu.au}

\cortext[cor1]{Corresponding author}
\address[RMS]{Department of Mathematics, University of Manitoba, Winnipeg, Canada}
\address[SO]{School of Mathematics and Statistics, The University of Sydney, Sydney, Australia.}

\begin{abstract}
We develop a spectral method for solving univariate singular integral equations over unions of intervals by utilizing Chebyshev and ultraspherical polynomials to reformulate the equations as almost-banded infinite-dimensional systems. This is accomplished by utilizing low rank approximations for sparse representations of the bivariate kernels. The resulting system can be solved in ${\cal O}(m^2n)$ operations using an adaptive QR factorization, where $m$ is the bandwidth and $n$ is the optimal number of unknowns needed to resolve the true solution.  The complexity is reduced to ${\cal O}(m n)$ operations by pre-caching the QR factorization when the same operator is used for multiple right-hand sides.  Stability is proved by showing that the resulting linear operator can be diagonally preconditioned to be a compact perturbation of the identity.  Applications considered include the Faraday cage, and acoustic scattering for the Helmholtz and gravity Helmholtz equations, including spectrally accurate numerical evaluation of the far- and near-field solution. The {\sc Julia} software package {\tt SingularIntegralEquations.jl} implements our method with a convenient, user-friendly interface.
\end{abstract}


\end{frontmatter}


\section{Introduction}

Singular integral equations are prevalent in the study of fracture mechanics~\cite{Erdogan-37-171-00}, acoustic scattering problems~\cite{Colton-Kress-83,Kress-61-345-95,Huybrechs-Vandewalle-29-2305-07,Kress-10,Hewett-Langdon-Melenk-51-629-13}, Stokes flow~\cite{Young-et-al-211-1-06}, Riemann--Hilbert problems~\cite{Olver-122-305-12}, and beam physics~\cite{Constantinou-Thesis-91,Heller-91}. We develop a fast and stable algorithm for the solution of univariate singular integral equations of general form~\cite{Muskhelishvili-53}
\begin{equation}\label{eq:SIE}
\xint_\Gamma K(x,y)u(y){\rm\,d}y = f(x),\quad{\rm for}\quad x\in\Gamma,\qquad {\cal B}u = {\bf c},
\end{equation}
where $K(x,y)$ is singular along the line $y=x$, the $\times$ in the integral sign denotes either the Cauchy principal value or the Hadamard finite-part, $\Gamma$ is a union of bounded smooth open arcs in $\mathbb{R}^2$, and ${\cal B}$ is a list of functionals. To be precise, we consider the prototypical singular integral equations on $[-1,1]$ given by:
$$
\dfrac{1}{\pi}\ddashint_{-1}^1 \left(\dfrac{K_1(x,y)}{(y-x)^2} + \dfrac{K_2(x,y)}{y-x} + \log|y-x| K_3(x,y) + K_4(x,y)\right) u(y){\rm\,d}y = f(x),\quad{\rm for}\quad x\in[-1,1],
$$
where $K_1,\ldots,K_4$ are continuous bivariate kernels.

In this work, we use several remarkable properties of Chebyshev polynomials including their spectral convergence, explicit formul\ae~for their Hilbert and Cauchy transforms, and low rank bivariate approximations to construct a fast and well-conditioned spectral method for solving univariate singular integral equations. Chebyshev and ultraspherical polynomials are utilized to convert singular integral operators into numerically banded infinite-dimensional operators. To represent bivariate kernels, we use the low rank approximations of~\cite{Townsend-Trefethen-35-C495-13}, where expansions in Chebyshev polynomials are constructed via sums of outer products of univariate Chebyshev expansions. The minimal solution to the recurrence relation is automatically revealed by the adaptive QR factorization of~\cite{Olver-Townsend-55-462-13}. Diagonal right preconditioners are derived for integral equations encoding Dirichlet and Neumann boundary conditions such that the preconditioned operators are compact perturbations of the identity.

The inspiration behind the proposed numerical method is the ultraspherical spectral method for solving ordinary differential equations~\cite{Olver-Townsend-55-462-13}, where ordinary differential equations are converted to infinite-dimensional almost banded linear systems (an almost banded operator is a banded operator apart from a finite number of dense rows). These systems can be solved in infinite-dimensions, i.e., without truncating the operators~\cite{Olver-Townsend-57-14}, as implemented in {\tt ApproxFun.jl}~\cite{ApproxFun} in the {\sc Julia} programming language~\cite{Julia-12,Julia-14}. The {\sc Julia} software package {\tt SingularIntegralEquations.jl}~\cite{SingularIntegralEquations} implements our method with a convenient, user-friendly interface. As an extension of this framework for infinite-dimensional linear algebra, mixed equations involving derivatives and singular integral operators can be solved in a unified way.

\bigskip

Several classical numerical methods exist for singular Fredholm integral equations of the first kind. These include: the Nystr\"om method~\cite{Berthold-Junghanns-30-1351-93,Bruno-Lintner-12,Lintner-Bruno-12}, whereby integral operators are approximated by quadrature rules; the collocation method~\cite{Elliott-13-1041-82,Elliott-19-816-82}, where approximate solutions in a finite-dimensional subspace are required to satisfy the integral equation at a finite number of collocation points; and the Galerkin method~\cite{Galerkin-19-897-15,Smigaj-et-al-41-1-15}, where the approximate solution is sought from an orthogonal subspace and is minimal in the energy norm. The use of hybrid Gauss-trapezoidal quadrature rules~\cite{Kress-15-229-91,Alpert-20-1551-99,Helsing-Ojala-227-2899-08,Hao-Barnett-Martinsson-Young-40-245-14} can significantly increase the convergence rates when treating weakly singular kernels.

Numerous methods have exploited the underlying structure of the linear systems arising from discretizing integral equations. The most celebrated of these is the Fast Multipole Method of Greengard and Rokhlin~\cite{Greengard-Rokhlin-73-325-87}. Other characterizations in terms of semi-separability or other hierarchies have also gained prominence~\cite{Hackbusch-Nowak-54-463-89,Ambikasaran-Darve-57-477-13,Aminfar-Ambikasaran-Darve-14}. Exploiting the matrix structure allows for fast matrix-vector products, which then allows for Krylov subspace methods~\cite{Krylov-7-491-31} to be extremely competitive. For scattering of the Helmholtz equation in very special geometries, hybrid numerical-asymptotic methods have been derived for frequency-independent solutions to the Dirichlet and Neumann problems~\cite{Huybrechs-Vandewalle-29-2305-07,Martinsson-Rokhlin-221-288-07,Hewett-Langdon-Melenk-51-629-13,Hewett-Langdon-Chandler-Wilde-14}.

Previous works on Chebyshev-based methods for singular integral equations include Frenkel~\cite{Frenkel-51-326-83}, which derives recurrence relations for the Chebyshev expansion of a singular integral equation after expanding the bivariate kernel in a basis of Chebyshev polynomials of the first kind in both variables, and Chan et al.~\cite{Chan-Thesis-01,Chan-Fannjiang-Paulino-41-683-03} in fracture mechanics, among others.  A similar analysis in~\cite{Frenkel-51-335-83} is used for hypersingular integrodifferential equations by expanding the bivariate kernel in a basis of Chebyshev polynomials of the second kind. This paper is an extension of these ideas with essential practical numerical considerations.

\begin{remarks}~
\begin{enumerate}
\item Combined with fast multiplication of Chebyshev series, our method is suitable for use in iterative Krylov subspace methods.
\item There is a great diversity of integral equation formulations. The choice of formulation depends on many properties, including for example, whether the boundary is open or closed and whether there are resonances. Most equations involve operators that contain manipulations of the fundamental solution, which would still satisfy the requirements of our method. However, we focus on the direct integral equations to retain a simple exposition.
\end{enumerate}
\end{remarks}

\section{Boundary integral equations in two dimensions}

In two dimensions, let ${\bf x} = (x_1,x_2)$ and ${\bf y} = (y_1,y_2)$. Positive definite second-order linear elliptic partial differential operators (PDOs) with variable coefficients are always reducible to the following {\it canonical form}~\cite{Vekua-67}:
\begin{equation}\label{eq:PDO}
{\bf L}\{u\} = \Delta u +a \dfrac{\partial u}{\partial x_1} + b\frac{\partial u}{\partial x_2} + cu.
\end{equation}
Let $\Phi({\bf x},{\bf y})$ denote the positive definite {\it fundamental solution} of~\eqref{eq:PDO} satisfying the formal partial differential equation (PDE)
\begin{equation}
{\bf L_x}\{\Phi\} = -\delta({\bf x}-{\bf y}),
\end{equation}
where $\delta$ is the two-dimensional Dirac delta distribution and the subscript indicates that ${\bf L}$ is acting in the ${\bf x}$ variable.

\subsection{Exterior scattering problems}

Let $\Gamma$ be a union of disjoint bounded smooth open arcs in $\mathbb{R}^2$ and let $\Omega:=\mathbb{R}^2\setminus\overline{\Gamma}$.

Let $\hat{u}({\bf x}) := \dfrac{1}{2\pi}\int_{\mathbb{R}^2} e^{-{\rm i}{\bf x}\cdot{\bf y}}u({\bf y}){\rm\,d}{\bf y}$ be the standard Fourier transform in $\mathbb{R}^2$. Then for $s\in\mathbb{R}$, $H^s(\mathbb{R}^2)$ defines the Bessel potential space as the set of normed tempered distributions equipped with:
\begin{equation}
\left\|u\right\|_{H^s(\mathbb{R}^2)}^2 = \int_{\mathbb{R}^2}(1+|{\bf x}|^2)^s|\hat{u}({\bf x})|^2{\rm\,d}{\bf x}.
\end{equation}
For bounded $\Omega$ with non-empty interior, the space $H^s(\Omega) := \left\{u|_\Omega:u\in H^s(\mathbb{R}^2)\right\}$. Furthermore, let $H^s_{\rm loc}(\Omega)$ denote the spaced of locally integrable functions in $H^s(\Omega)$ and lastly, let $\tilde{H}^s(\Omega) :=\left\{u\in H^s(\mathbb{R}^2):{\rm\,supp\,} u \subset \overline{\Omega}\right\}$.

\begin{definition}
For ${\bf x}\in \Omega$, let ${\cal S}_\Gamma$ and ${\cal D}_\Gamma$ define the single- and double-layer potentials:
\begin{align}
{\cal S}_\Gamma u({\bf x}) & = \int_\Gamma \Phi({\bf x},{\bf y})u({\bf y}){\rm\,d}\Gamma({\bf y}) : \tilde{H}^{-\frac{1}{2}}(\Gamma) \to H^1_{\rm loc}(\Omega),\\
{\cal D}_\Gamma u({\bf x}) & = \int_\Gamma \dfrac{\partial\Phi({\bf x},{\bf y})}{\partial n({\bf y})}u({\bf y}){\rm\,d}\Gamma({\bf y}) : \tilde{H}^{\frac{1}{2}}(\Gamma) \to H^1_{\rm loc}(\Omega).
\end{align}
\end{definition}

\begin{definition}[Radiation condition at infinity~\cite{Costabel-Dauge-20-133-97}]
We say that $u\in H^1_{\rm loc}(\Omega)$ satisfies the {\em radiation condition at infinity} if:
\begin{equation}
\lim_{\rho\to+\infty}\left\{ {\cal S}_{|{\bf y}| = \rho}\left[\partial u/\partial n\right]({\bf x}) - {\cal D}_{|{\bf y}| = \rho}\left[u\right]({\bf x}) \right\} = 0,\quad{\rm for}\quad {\bf x}\in \Omega.
\end{equation}
\end{definition}

For solutions of the homogeneous equation ${\bf L}\{u\} = 0$ satisfying the radiation condition at infinity, Green's representation theorem allows for the determination of the exterior solutions given data on the boundary $\Gamma$:
\begin{equation}\label{eq:Greenstheorem}
u({\bf x}) = -{\cal S}_\Gamma\left[\partial u/\partial n\right]({\bf x}) + {\cal D}_\Gamma\left[u\right]({\bf x}),\quad{\rm for}\quad {\bf x}\in \Omega.
\end{equation}
Here, $[u]$ denotes the jump in $u$ along $\Gamma$ and $[\partial u/\partial n]$ the jump in its normal derivative. These are formally defined by the Dirichlet trace and conormal derivative~\cite{Smigaj-et-al-41-1-15}, or in the case of the Laplace equation, simply as the difference between the limiting values on $\Gamma$ as we approach from the left and the right. This identity can be interpreted as representing $u$ in terms of the potential of a distribution of poles on $\Gamma$ through the single-layer and normal dipoles on $\Gamma$ through the double layer. With either Dirichlet or Neumann boundary conditions, we restrict~\eqref{eq:Greenstheorem} to the boundary and solve for the unknown boundary value. Once both quantities on the boundary are determined, the solution to the exterior problem is readily available in integral form.

{\bf Dirichlet Problem}
Given an incident wave $u^{i}({\bf x})\in H^{\frac{1}{2}}(\Gamma)$ satisfying ${\bf L}\{u^i\} = 0$, find $u^s({\bf x}) \in H^1_{\rm loc}(\Omega)$ satisfying ${\bf L}\{u^s\} = 0$, the radiation condition at infinity, and
\begin{equation}
u^i({\bf x}) + u^s({\bf x}) = 0,\quad{\rm for}\quad {\bf x}\in\Gamma.
\end{equation}

{\bf Neumann Problem}
Given $\dfrac{\partial u^{i}({\bf x})}{\partial n({\bf x})} \in H^{-\frac{1}{2}}(\Gamma)$ satisfying ${\bf L}\{u^i\} = 0$, find $u^s({\bf x}) \in H^1_{\rm loc}(\Omega)$ satisfying ${\bf L}\{u^s\} = 0$, the radiation condition at infinity, and
\begin{equation}
\dfrac{\partial}{\partial n({\bf x})}\left(u^i({\bf x}) + u^s({\bf x})\right) = 0,\quad{\rm for}\quad {\bf x}\in\Gamma.
\end{equation}

For the case of the Laplace and Helmholtz equations, the Dirichlet problem is originally formulated in~\cite[Eqs. (1.1) \& (1.2), (1.6) \& (1.7)]{Stephan-Wendland-18-183-84}. Similarly, the Neumann problem is originally formulated in~\cite[Eqs. (1.1) \& (1.2)]{Wendland-Stephan-112-363-90}.

{\bf Dirichlet Solution}
The Dirichlet problem is solved by~\eqref{eq:Greenstheorem} where $[u^s] = 0$, and the scattered solution is represented everywhere by the single-layer potential. The density $\left[\partial u^s/\partial n\right]\in \tilde{H}^{-\frac{1}{2}}(\Gamma)$ in~\eqref{eq:Greenstheorem} satisfies:
\begin{equation}\label{eq:DirichletIE}
\int_\Gamma \Phi({\bf x},{\bf y})\left[\dfrac{\partial u^s}{\partial n}\right]{\rm d}\Gamma({\bf y}) = u^i({\bf x}),\quad{\rm for}\quad {\bf x}\in\Gamma.
\end{equation}

{\bf Neumann Solution}
The Neumann problem is solved by~\eqref{eq:Greenstheorem} where $\left[\partial u^s/\partial n\right] = 0$, and the scattered solution is represented everywhere by the double-layer potential. The density $\left[u^s\right]\in\tilde{H}^{\frac{1}{2}}(\Gamma)$ in~\eqref{eq:Greenstheorem} satisfies:
\begin{equation}\label{eq:NeumannIE}
\dfrac{\partial}{\partial n({\bf x})}\int_\Gamma \dfrac{\partial \Phi({\bf x},{\bf y})}{\partial n({\bf y})}\left[u^s\right]{\rm d}\Gamma({\bf y}) = -\dfrac{\partial u^i({\bf x})}{\partial n({\bf x})},\quad{\rm for}\quad {\bf x}\in\Gamma.
\end{equation}

For the case of the Laplace and Helmholtz equations, the Dirichlet solution is originally proved in~\cite[Theorems 1.4 \& 1.7]{Stephan-Wendland-18-183-84}. Similarly, the Neumann solution is originally proved in~\cite[Theorem 1.3]{Wendland-Stephan-112-363-90}. Furthermore, by appealing to the theory of Mellin transforms, inverse square root singular behaviour is derived for the open ends of $\Gamma$ in the Dirichlet problem~\cite[Theorem 2.3]{Stephan-Wendland-18-183-84}, and square root singular behaviour is derived for the open ends of $\Gamma$ in the Neumann problem~\cite[Theorem 1.8]{Wendland-Stephan-112-363-90}.

For elliptic PDOs with variable coefficients, we use the solutions provided by the singular integral equations~\eqref{eq:DirichletIE} and~\eqref{eq:NeumannIE}, and while the scope of this paper is {\em numerical}, we conjecture that they hold more generally.

\subsection{Riemann functions}

In addition to the PDO in~\eqref{eq:PDO}, consider its adjoint:
\begin{equation}
{\bf L}^*\{v\} = \Delta v -\dfrac{\partial (av)}{\partial x_1} - \frac{\partial (bv)}{\partial x_2} + cv.
\end{equation}
With the change to complex characteristic variables:
\begin{equation}
z = x_1+\I x_2,\qquad \zeta = x_1-\I x_2, \qquad \qquad z_0 = y_1+\I y_2,\qquad \zeta_0 = y_1-\I y_2,
\end{equation}
${\bf L}$ and ${\bf L}^*$ take the form:
\begin{align}
\hat{\bf L}\{U\} & = \dfrac{\partial^2U}{\partial z\partial\zeta} + A\dfrac{\partial U}{\partial z} + B\dfrac{\partial U}{\partial \zeta} + CU,\\
\hat{\bf L}^*\{V\} & = \dfrac{\partial^2V}{\partial z\partial\zeta} - \dfrac{\partial (AV)}{\partial z} - \dfrac{\partial (BV)}{\partial \zeta} + CV,
\end{align}
where:
\begin{align}
A(z,\zeta) & = \dfrac{1}{4}\left[a\left(\frac{z+\zeta}{2},\frac{z-\zeta}{2\I}\right) + \I b\left(\frac{z+\zeta}{2},\frac{z-\zeta}{2\I}\right)\right],\label{eq:ellipticA}\\
B(z,\zeta) & = \dfrac{1}{4}\left[a\left(\frac{z+\zeta}{2},\frac{z-\zeta}{2\I}\right) - \I b\left(\frac{z+\zeta}{2},\frac{z-\zeta}{2\I}\right)\right],\label{eq:ellipticB}\\
C(z,\zeta) & = \dfrac{1}{4}c\left(\frac{z+\zeta}{2},\frac{z-\zeta}{2\I}\right).\label{eq:ellipticC}
\end{align}
\begin{theorem}[Vekua and Garabedian~\cite{Vekua-67,Garabedian-64}]
For analytic functions~\eqref{eq:ellipticA}--\eqref{eq:ellipticC}, there exist analytic functions $\mathfrak{R}(z,\zeta,z_0,\zeta_0)$ and $g_0(z,\zeta,z_0,\zeta_0)$ such that:
\begin{equation}\label{eq:GFcharacteristics}
\Phi(z,\zeta,z_0,\zeta_0) = -\dfrac{1}{4\pi}\mathfrak{R}(z,\zeta,z_0,\zeta_0)\log[(z-z_0)(\zeta-\zeta_0)] + g_0(z,\zeta,z_0,\zeta_0),
\end{equation}
where $\hat{\bf L}\{\Phi\} = 0$ in $(z,\zeta)$ and $\hat{\bf L}^*\{\Phi\} = 0$ in $(z_0,\zeta_0)$ so long as $z\ne z_0$ and $\zeta\ne\zeta_0$. In~\eqref{eq:GFcharacteristics}, $\mathfrak{R}$ is the Riemann function of the operator ${\bf L}$ satisfying:
\begin{align}
\hat{\bf L}^* \{\mathfrak{R}\} & = 0,\label{eq:RiemannL}\\
\mathfrak{R}(z_0,\zeta,z_0,\zeta_0) & = \exp\left\{ \int_{\zeta_0}^\zeta A(z_0,\tau){\rm\,d}\tau\right\},\quad{\rm and}\\\mathfrak{R}(z,\zeta_0,z_0,\zeta_0) & = \exp\left\{ \int_{z_0}^z B(t,\zeta_0){\rm\,d}t\right\}.\label{eq:Riemannbcs}
\end{align}  
\end{theorem}

\begin{remarks}~
It is straightforward to reformulate~\eqref{eq:RiemannL}--\eqref{eq:Riemannbcs} to the following integral equation:
\begin{align}
\mathfrak{R}(z,\zeta,z_0,\zeta_0) - \int_{z_0}^z B(t,\zeta)\mathfrak{R}(t,\zeta&,z_0,\zeta_0){\rm\,d}t - \int_{\zeta_0}^\zeta A(z,\tau)\mathfrak{R}(z,\tau,z_0,\zeta_0){\rm\,d}\tau\nonumber\\
& + \int_{z_0}^z\int_{\zeta_0}^\zeta C(t,\tau)\mathfrak{R}(t,\tau,z_0,\zeta_0){\rm\,d}\tau{\rm\,d}t = 1.\label{eq:RiemannIE}
\end{align}

Returning to the original coordinates ${\bf x}$ and ${\bf y}$, fundamental solutions for elliptic PDOs with analytic coefficients can be written as:
\begin{equation}\label{eq:GFisAlogplusB}
\Phi({\bf x},{\bf y}) = A({\bf x},{\bf y})\log\left|{\bf x}-{\bf y}\right| + B({\bf x},{\bf y}),
\end{equation}
where $A$ and $B$ are both analytic functions of ${\bf x}$ and ${\bf y}$ and where $A({\bf x},{\bf y}) = -\frac{1}{2\pi}\mathfrak{R}(z,\zeta,z_0,\zeta_0)$ implying $A({\bf x},{\bf x}) = -(2\pi)^{-1}$. If, furthermore, the PDO is formally self-adjoint, then $A$ and $B$ are also symmetric functions of ${\bf x}$ and ${\bf y}$.
\end{remarks}

\section{Practical approximation theory}

Chebyshev approximation theory is a very rich subject that has seen numerous exceptional contributions: see~\cite{Boyd-00,Mason-Handscomb-02,Trefethen-12} and the references therein. In this section, we describe some approximation spaces for one-dimensional intervals and two-dimensional squares. For every approximation space, one may consider the interpolants, which are equal to the function at a set of interpolation points, and the projections, which are truncations of the function's expansion. Unless an extraordinary amount of analytic information is known about a function, interpolants are generally easier to construct.

We consider an approximation space {\it practical} if there is a fast way to transform the interpolation condition into approximate projections. While a few methods exist to create fast transforms, all the practical approximation spaces we consider resort to some variation of the fast Fourier transform (FFT)~\cite{Cooley-Tukey-19-297-65,Frigo-Johnson-93-216-05} to reduce ${\cal O}(n^2)$ complexity to ${\cal O}(n\log n)$. Other properties which make an approximation space practical are: ${\cal O}(n)$ evaluation; a low Lebesgue constant; absolute, uniform, and geometric convergence with analyticity; and, easy manipulation for the development of new properties. For approximation on the canonical unit interval~~$\mathbb{I} := [-1,1]$, we will make our statements precise in the following subsection.

\subsection{One dimension}

Let $K$ be the field of $\mathbb{R}$ or $\mathbb{C}$. A function $f:\mathbb{I}\to K$ is of bounded total variation if:
\begin{equation}
V_f = \int_{\mathbb{I}} |f'(z)|{\rm\,d}z < +\infty.
\end{equation}

Chebyshev polynomials of the first kind are defined by~\cite{Mason-Handscomb-02}:
\begin{equation}
T_n(x) = \cos(n\cos^{-1}(x)), \quad{\rm for}\quad n\in\mathbb{N}_0,\quad{\rm and}\quad x\in\mathbb{I}.
\end{equation}
A Chebyshev interpolant to a continuous function $f:\mathbb{I}\to K$ is the approximation
\begin{equation}
p_N(x) = \sum_{n=0}^{N-1} c_n T_n(x),\quad x\in\mathbb{I},
\end{equation}
which interpolates $f$ at the Chebyshev points of the first kind:
\begin{equation}
p_N(x_n) = f(x_n)\quad{\rm where}\quad x_n = \cos\left(\dfrac{2n+1}{2N}\pi\right),\quad{\rm for}\quad n=0,\ldots,N-1.
\end{equation}

The Chebyshev basis has fast transforms between values at Chebyshev points and coefficients via fast implementations of the discrete cosine transforms (DCTs). The (orthogonal) Chebyshev polynomials satisfy a three-term recurrence relation that can be used in Clenshaw's algorithm~\cite{Clenshaw-9-118-55} for ${\cal O}(n)$ evaluation of interpolants. Compared with the best polynomial approximants, Chebyshev interpolants are near-best in the sense that their Lebesgue constants exhibit similar logarithmic growth.

\begin{theorem}[Battles and Trefethen~\cite{Battles-Trefethen-25-1743-04}]
Let $f$ be a continuous function on $\mathbb{I}$, $p_N$ its $N$-point polynomial interpolant in the Chebyshev points of the first kind and $p_N^\star$ its best degree-$N-1$ polynomial approximation. Then:
\begin{enumerate}
\item $\|f - p_N \|_\infty \le \left(2+\frac{2}{\pi}\log N-1\right)\|f - p_N^\star\|_\infty$;
\item if $f$ has a $k^{\rm th}$ derivative in $\mathbb{I}$ of bounded variation for some $k\ge1$, $\|f-p_N\|_\infty = {\cal O}(N^{-k})$ as $N\to\infty$; and,
\item if $f$ is analytic in a neighbourhood of $\mathbb{I}$, $\|f - p_N\|_\infty = {\cal O}(C^N)$ as $N\to\infty$ for some $C < 1$; in particular we may take $C = 1/(M +m)$ if $f$ is analytic in the closed Bernstein ellipse with foci~$\pm1$ and semimajor and semiminor axis lengths $M \ge1$ and $m\ge0$.
\end{enumerate}
\end{theorem}

An interpolant can be constructed to any relative or absolute tolerance $\epsilon$ by successively doubling the number of interpolation conditions, transforming values to coefficients, and determining an acceptable degree\footnote{This heuristic determination is usually based on, among other things, the relative and absolute magnitudes of initial and final coefficients, the decay rate of the coefficients, an estimate of the condition number of the function, and an estimate of the Lebesgue constant for a given degree.}.

\subsection{Two dimensions}

Numerous methods have been devised to approximate functions in more than one dimension. The straightforward generalization of the one-dimensional approach is to sample the function on a tensor of one-dimensional interpolation points and to adaptively truncate coefficients below a certain threshold.

Consider the function $f:\mathbb{I}^2\to K$, whose two-dimensional Chebyshev interpolant takes the form:
\begin{equation}\label{eq:ProductFun}
p_{m,n}(x,y) = \sum_{i=0}^{m-1}\sum_{j=0}^{n-1} A_{i,j} T_i(x) T_j(y).
\end{equation}
While the tensor approach in general suffers from the curse of dimensionality, it can still be competitive in two dimensions, scaling with ${\cal O}(mn)$ function samples and ${\cal O}(\min(mn\log n,nm\log m))$ arithmetic via fast two-dimensional transforms.

The singular value decomposition of an $m\times n$ matrix ${\bf A}$ over $K$ is the factorization~\cite{Watkins-10}:
\begin{equation}
{\bf A} = {\bf U}\boldsymbol{\Sigma}{\bf V^*},
\end{equation}
where ${\bf U}$ is an $m\times m$ unitary matrix over $K$, $\boldsymbol{\Sigma}$ is an $m\times n$ diagonal matrix of non-negative {\it singular values}, and ${\bf V^*}$ is an $n\times n$ unitary matrix over $K$. The singular value decomposition reveals the rank of a matrix as the number of nonzero singular values.

If we perform the singular value decomposition of the matrix of coefficients in~\eqref{eq:ProductFun}, the approximation to $f$ can be re-expressed as:
\begin{equation}\label{eq:SVDFun}
p_{\rm SVD}(x,y) = \sum_{i=1}^k \sigma_i u_i(x) v^*_i(y),
\end{equation}
where $\sigma_i$ are the singular values, and $u_i(x)$ and $v^*_i(y)$ are univariate Chebyshev approximants with coefficients from the columns of ${\bf U}$ and the rows of ${\bf V^*}$, respectively, and where ${\bf A}$ is of rank $k$. It follows that $p_{\rm SVD}$ is the best rank-$k$ approximant in $L^2(\mathbb{I}^2)$ to $f$ that can be obtained for the original two-dimensional interpolant. For any given tolerance $\epsilon>0$, a function $f$ has numerical rank $k_\epsilon$ if~\cite{Townsend-Thesis-14}
\begin{equation}
k_\epsilon = \inf_{k\in\mathbb{N}}\left\{ \inf_{f_k}\| f-f_k\|_\infty \le \epsilon \|f\|_\infty\right\},
\end{equation}
where the inner infimum is taken over all rank-$k$ functions.
\begin{definition}[Townsend Definition 3.1~\cite{Townsend-Thesis-14}]
For some $\epsilon>0$, let $k_\epsilon$ be the numerical rank of $f : \mathbb{I}^2\to K$, and $m_\epsilon$ and $n_\epsilon$ be the maximal degrees of the univariate approximations in the $x$ and $y$ variables. If $k_\epsilon(m_\epsilon + n_\epsilon) < m_\epsilon n_\epsilon$, we say the function $f$ is numerically of {\em low rank}, and if $k_\epsilon\approx \min(m_\epsilon,n_\epsilon)$, then the function $f$ is numerically of {\em full rank}.
\end{definition}

A particularly attractive scheme for calculating low rank approximation in two dimensions can be described as a continuous analogue of Gaussian elimination~\cite{Townsend-Thesis-14} and is a direct extension of the greedy algorithm in one dimension~\cite[Chapter 5]{Trefethen-12}. This algorithm is studied in depth in Townsend's DPhil thesis and implementations are found in {\tt Chebfun}~\cite{Chebfun} and {\tt ApproxFun.jl}~\cite{ApproxFun}. In this algorithm, the function is initially sampled on a grid to locate its approximate absolute maximum. Two one-dimensional approximations are created in the $x$ and $y$ variables to interpolate the function along the row and column that intersect at the approximate absolute maximum. After subtracting this rank-one approximation, the algorithm continues its search for the next approximate absolute maximum. After $k$ iterations, it is clear that the approximant
\begin{equation}\label{eq:LowRankFun}
p_{\rm GE}(x,y) = \sum_{i=1}^k A_i(x) B_i(y),
\end{equation}
coincides with $f$ in the $k$ rows and columns whose intersections coincide with an iteration's approximate absolute maximum. As the size of the sampling grid increases, the approximate absolute maxima will converge to the true absolute maxima and in this sense we reproduce close aproximations to $p_{\rm SVD}$. In terms of the degrees of the one-dimensional approximations $m,n$ and the rank $k$, the algorithm scales with a search over ${\cal O}(mn)$ function samples and ${\cal O}(k\,(m\log m+n\log n))$ arithmetic via fast one-dimensional transforms.

\begin{definition}[Townsend Definition 4.11~\cite{Townsend-Thesis-14}]
The function $f:\mathbb{I}^2\to K$ is Hermitian if it satisfies the conjugate symmetry $f(x,y) = f^*(y,x)$ and it is non-negative definite, i.e.:
\begin{equation}
\iint_{\mathbb{I}^2} a^*(y)f(y,x)a(x){\rm\,d}y{\rm\,d}x \ge 0,
\end{equation}
for all $a(x)\in C(\mathbb{I})$.
\end{definition}

When a bivariate function is Hermitian, even further savings can be obtained by drawing the analogy to the Cholesky factorization of a Hermitian matrix~\cite{Townsend-Trefethen-471-20140585-15}:
\begin{equation}
p_{\rm Cholesky}(x,y) = \sum_{i=1}^kA_i(x)A_i^*(y).
\end{equation}
In this case, it is known that the function's absolute maxima after every iteration are on the diagonal line $y=x$, leading to a reduction in the dimension of the search space. In addition, as they are conjugates only either the row or column slices may be computed and stored.

\subsection{An algorithm to extract the splitting of a fundamental solution}

Accurate numerical evaluation of a fundamental solution on or near the singular diagonal may not always be possible or may be more expensive~\cite{Barnett-Nelson-Mahoney-297-407-15}. To avoid the numerical problems associated with the singular diagonal, we use Chebyshev points of the first kind in one direction and Chebyshev points of the second kind~\cite{Mason-Handscomb-02} in the other direction. This ensures that the diagonal is never sampled. In terms of the DCTs, taking $2^n$ points of the first kind is optimal and taking $2^n+1$ points of the second kind is nearly optimal.

When both $A({\bf x},{\bf y})$ and $B({\bf x},{\bf y})$ in~\eqref{eq:GFisAlogplusB} are not known {\it a priori}, but the fundamental solution itself can be evaluated, we can use such skewed grids in combination with the Riemann function $\mathfrak{R}$ to:
\begin{enumerate}
\item approximate $A({\bf x},{\bf y}) \equiv -\frac{1}{2\pi}\mathfrak{R}(x_1+\I x_2,x_1-\I x_2,y_1+\I y_2,y_1-\I y_2)$; and subsequently,
\item approximate the difference $B({\bf x},{\bf y}) \equiv \Phi({\bf x},{\bf y}) -A({\bf x},{\bf y})\log|{\bf x}-{\bf y}|$.
\end{enumerate}

\section{The ultraspherical spectral method}

The ultraspherical spectral method of Olver and Townsend~\cite{Olver-Townsend-55-462-13} represents solutions of linear ordinary differential equations of the form
\begin{equation}
{\cal A} u = f, \qquad {\cal B}u = c,
\end{equation}
where ${\cal A}$ is a linear operator of the form
\begin{equation}\label{eq:LinearOperator}
{\cal A} = a_N(x) \dfrac{{\rm d}^N}{{\rm d}x^N} + \cdots + a_1(x)\frac{\rm d}{{\rm d}x} + a_0(x),
\end{equation}
and ${\cal B}$ contains $N$ linear functionals.  Typically, ${\cal B}$ encodes boundary conditions such as Dirichlet or Neumann conditions.  We consider $u(x)$ in its Chebyshev expansion
\begin{equation}
u(x) = \sum_{n=0}^\infty u_n T_n(x),
\end{equation}
so that $u(x)$ can be identified by a vector of its Chebyshev coefficients  ${\bf u} = (u_0,u_1,\ldots)^\top$.

To solve such a problem efficiently, a change of basis occurs for each order of spectral differentiation, using the formula:
\begin{equation}
\dfrac{{\rm d}^\lambda T_n(x)}{{\rm d}x^\lambda} = \left\{ \begin{array}{cc} 0, & 0\le n\le \lambda-1,\\
2^{\lambda-1}(\lambda-1)!\,n\,C_{n-\lambda}^{(\lambda)}(x), & n\ge\lambda,
\end{array}\right.
\end{equation}
where $C_n^\lambda$ represents the ultraspherical polynomial of integral order $\lambda$ and of degree $n$. This sparse differentiation has the operator representation:
\begin{equation}
{\cal D}_\lambda = 2^{\lambda-1}(\lambda-1)!\begin{pmatrix} \overbrace{0~~\cdots~~0}^{\lambda~{\rm times}} & \lambda\\
&&\lambda+1\\
&&&\lambda+2\\
&&&&\ddots\end{pmatrix},\qquad \lambda\ge1,
\end{equation}
and maps the Chebyshev coefficients to the $\lambda^{\rm th}$ order ultraspherical coefficients.

Since in~\eqref{eq:LinearOperator}, each derivative maps to a different ultraspherical basis, the sparse differentiation operators are accompanied by sparse conversion operators such that ${\cal A}$ can be expressed completely in the basis of highest order $N$:
\begin{equation}
{\cal S}_0 = \begin{pmatrix} 1 & 0 & -\frac{1}{2}\\
& \frac{1}{2} & 0 & -\frac{1}{2}\\
& & \frac{1}{2} & 0 & \ddots\\
& & & \ddots & \ddots
\end{pmatrix},\qquad
{\cal S}_\lambda = \begin{pmatrix} 1 & 0 & -\frac{\lambda}{\lambda+2}\\
& \frac{\lambda}{\lambda+1} & 0 & -\frac{\lambda}{\lambda+3}\\
& & \frac{\lambda}{\lambda+2} & 0 & \ddots\\
& & & \ddots & \ddots
\end{pmatrix},\quad \lambda\ge1.
\end{equation}
Here, ${\cal S}_0$ maps the Chebyshev coefficients to the first order ultraspherical coefficients and ${\cal S}_\lambda$ maps the $\lambda^{\rm th}$ order ultraspherical coefficients to the $(\lambda+1)^{\rm th}$ order ultraspherical coefficients. Therefore, the conversion and differentiation operators can be combined in ${\cal A}$ as follows:
\begin{equation}\label{eq:LwithConstantCoefficients}
\left( a_N{\cal D}_N + a_{N-1} {\cal S}_{N-1}{\cal D}_{N-1} + \cdots + a_0 {\cal S}_{N-1}\cdots{\cal S}_0\right) {\bf u} = {\cal S}_{N-1}\cdots{\cal S}_0{\bf f},
\end{equation}
where ${\bf u}$ and ${\bf f}$ are vectors of Chebyshev expansion coefficients. Were the coefficients $a_i(x)$, $i=0,\ldots,N$, constant, then~\eqref{eq:LwithConstantCoefficients} would represent a linear recurrence relation in the coefficients ${\bf u}$ of length at most $2N+1$. However, the coefficients are in general not constants, so the multiplication operators in Chebyshev and ultraspherical bases are also investigated in~\cite{Olver-Townsend-55-462-13}. Let
\begin{equation}
a(x) = \sum_{n=0}^\infty a_n T_n(x).
\end{equation}
Then it is shown in~\cite{Olver-Townsend-55-462-13} that multiplication can be represented as a Toeplitz-plus-Hankel-plus-rank-one operator:
\begin{equation}
{\cal M}_0[a] = \dfrac{1}{2}\left[\begin{pmatrix} 2a_0 & a_1 & a_2 & \cdots\\ a_1 & 2a_0 & a_1 & \ddots\\ a_2 & a_1 & 2a_0 & \ddots\\ \vdots & \ddots & \ddots & \ddots\\\end{pmatrix} + \begin{pmatrix} 0 & 0 & 0 & \cdots\\ a_1 & a_2 & a_3 & \cdots\\ a_2 & a_3 & a_4 & \iddots\\ \vdots & \iddots & \iddots & \iddots\\\end{pmatrix} \right].
\end{equation}
For $\lambda>0$, an explicit formula for the entries is given in~\cite{Olver-Townsend-55-462-13} and a three-term recurrence relation is shown in~\cite[Chap. 6]{Townsend-Thesis-14}. By the associative and distributive properties of multiplication, the recurrence relation for the multiplication operators is derived from the recurrence relation for the ultraspherical polynomials:
\begin{equation}\label{eq:ultraMultrec}
{\cal M}_\lambda[C_{n+1}^{(\lambda)}] = \dfrac{2(n+\lambda)}{n+1}{\cal M}_\lambda[x]{\cal M}_\lambda[C_n^{(\lambda)}] - \dfrac{n+2\lambda-1}{n+1}{\cal M}_\lambda[C_{n-1}^{(\lambda)}],\qquad n\ge1.
\end{equation}
Since we assume the coefficients $a_i(x)$ to be continuous functions with bounded variation on $\mathbb{I}$, let $m$ denote the highest degree Chebyshev expansion such that for some $\epsilon>0$:
\begin{equation}
\left\|a_i(x) - \sum_{n=0}^{m-1}a_{in}T_n(x)\right\|_\infty \le \epsilon\|a_i(x)\|_\infty,\quad{\rm for}\quad i=0,\ldots,N.
\end{equation}
Then in this way, the system
\begin{align}
{\cal B}{\bf u} & = {\bf c},\nonumber\\
\left( {\cal M}_N[a_N]{\cal D}_N + {\cal M}_N[a_{N-1}] {\cal S}_{N-1}{\cal D}_{N-1} + \cdots + {\cal M}_N[a_0] {\cal S}_{N-1}\cdots{\cal S}_0\right) {\bf u} & = {\cal S}_{N-1}\cdots{\cal S}_0{\bf f},
\end{align}
is almost banded with bandwidth ${\cal O}(m)$. The proposed ${\cal O}(m^2n)$ solution process for such systems is the adaptive QR factorization, generalizing (F.~W.~J.)~Olver's algorithm for second-order difference equations~\cite{Olver-71B-111-67}. In this factorization, the forward error is estimated at every step in the infinite-dimensional upper-triangularization to adaptively determine the minimal order $n$ required to resolve the solution below a pre-determined accuracy. Since the unitary transformations implied by $Q$ preserve the rank structure, the back substitution is also performed with ${\cal O}(m^2n)$ complexity.

Figure~\ref{fig:Example0spyplot} shows the typical structure of the system and an example of the type of singularly perturbed boundary value problem that it can solve efficiently.

\begin{figure}[htbp]
\begin{center}
\begin{tabular}{cc}
\includegraphics[width=0.405\textwidth]{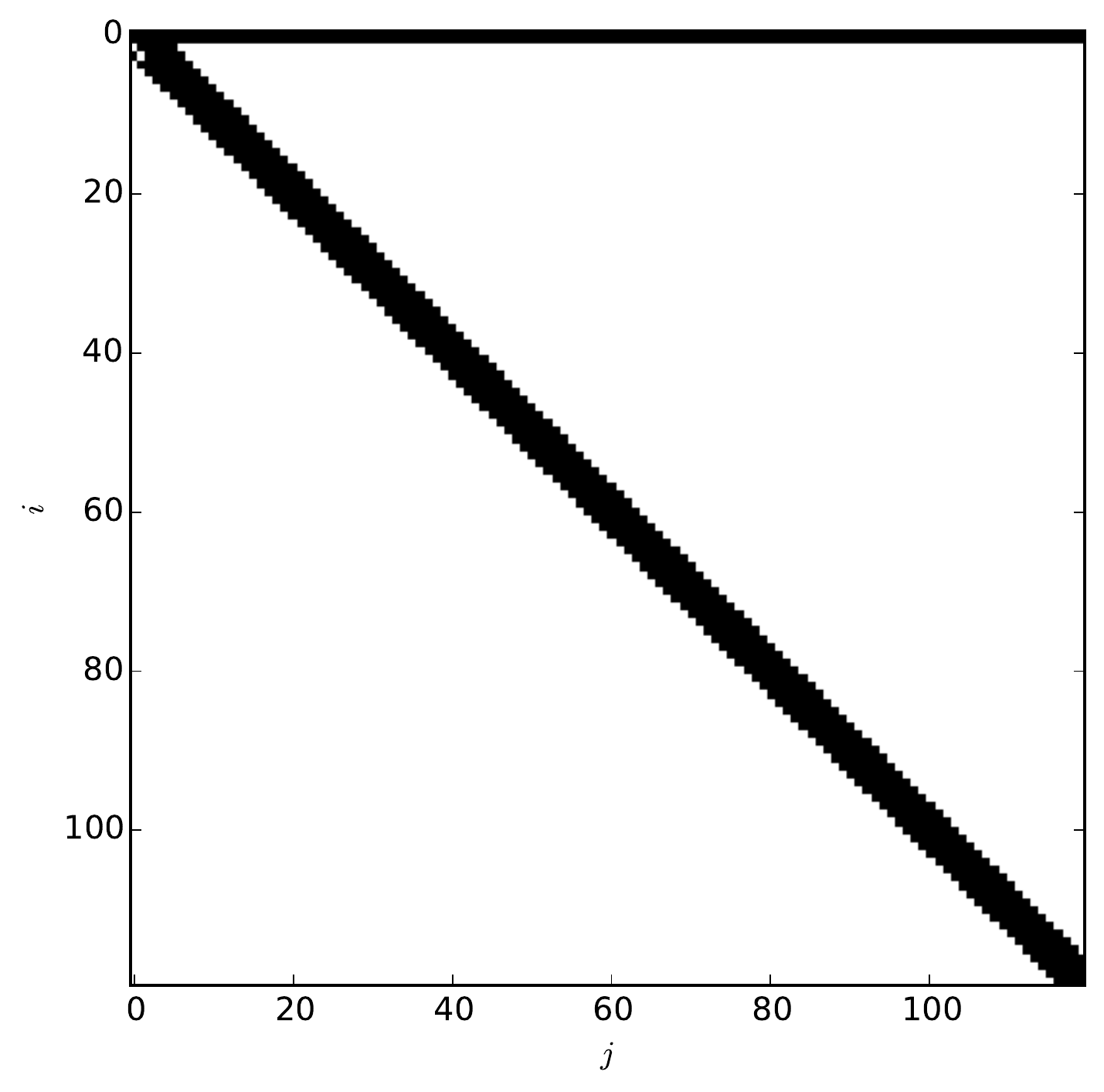}&
\includegraphics[width=0.57\textwidth]{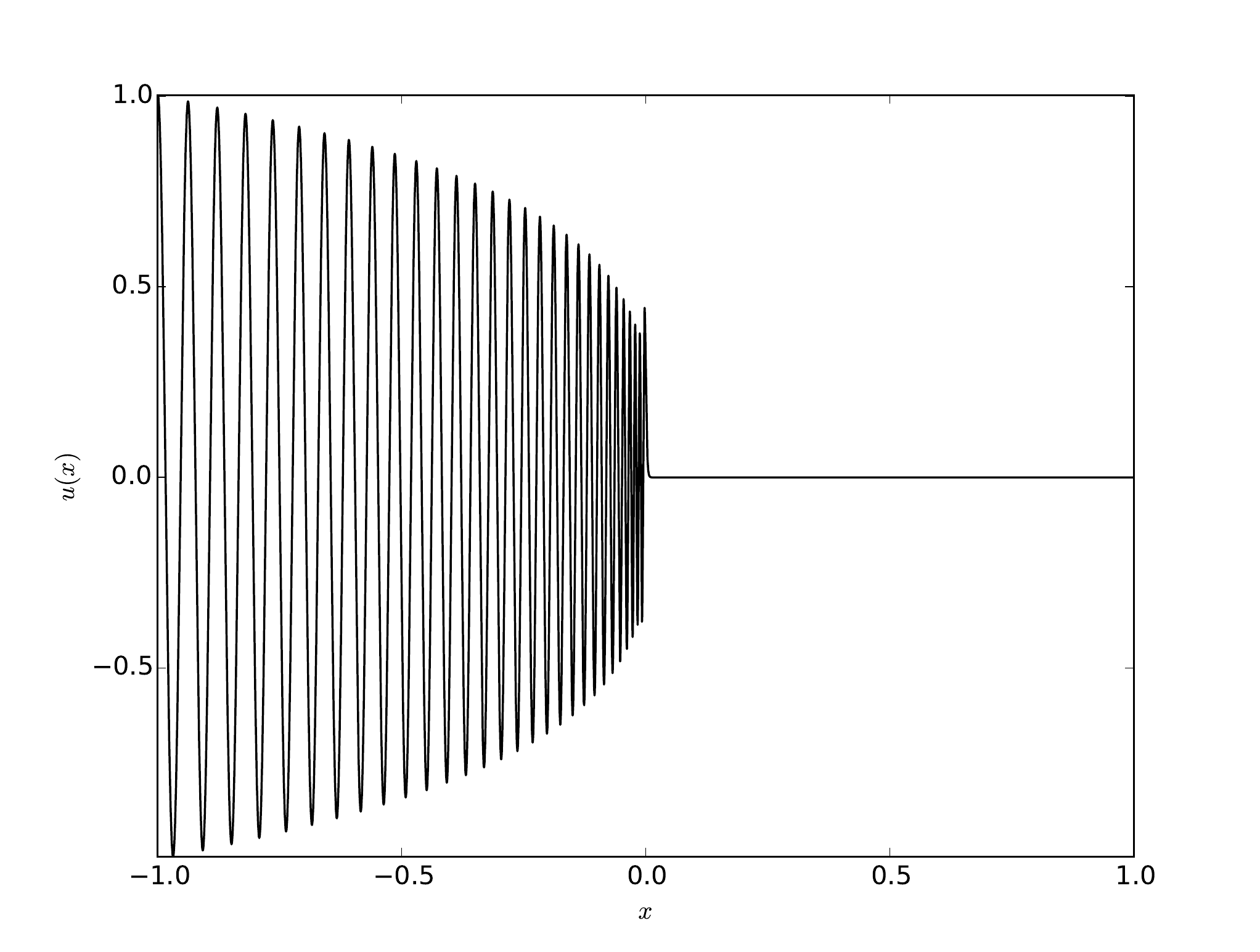}\\
\end{tabular}
\caption{Solution of $\epsilon(\epsilon+x^2)u''(x) = x\,u(x)$, $u(-1)=1$, $u(1)=0$ via the ultraspherical spectral method. Left: the structure of the system. Right: a plot of the solution for $\epsilon = 10^{-4}$. In this case, a Chebyshev expansion of degree $3,\!276$ is required to approximate the solution to double precision.}
\label{fig:Example0spyplot}
\end{center}
\end{figure}

\subsection{Almost-banded spectral methods in other bases}

The key elements of the ultraspherical spectral method are a graded set of bases that permit banded differentiation and conversion within the set of bases, and multiplication operators for variable coefficients. Other examples where a graded basis can be exploited are the Jacobi polynomials (which include Legendre and ultraspherical polynomials as special cases), and the generalized Laguerre polynomials. Hermite polynomials, which form an Appell sequence, satisfy $H'_n(x) = 2nH_{n-1}(x)$, and therefore do not require other bases for conversion.

From the three-term recurrence relation satisfied by orthogonal polynomials~\cite{Gautschi-04}:
\begin{equation}
x \pi_{n}(x) = \alpha_n \pi_{n+1}(x) + \beta_n \pi_{n}(x) + \gamma_n \pi_{n-1}(x),
\end{equation}
it follows that multiplication by $x$ is tridiagonal:
\begin{equation}
{\cal M}[x] = \begin{pmatrix} \beta_0 & \alpha_0  \\ \gamma_1 & \beta_1 & \alpha_1 &  &  \\  & \gamma_2 & \beta_2 & \alpha_2 &  \\  &  & \ddots & \ddots & \ddots \end{pmatrix}.
\end{equation}
Therefore, banded multiplication operators in orthogonal bases can be derived from the recurrence relation:
\begin{equation}
{\cal M}[\pi_{n+1}] = \left(\dfrac{{\cal M}[x] -\beta_n}{\alpha_n}\right){\cal M}[\pi_n] - \dfrac{\gamma_n}{\alpha_n}{\cal M}[\pi_{n-1}],\qquad n\ge1.
\end{equation}
and consequently variable coefficients represented as interpolants have a finite-bandwidth operator form. To numerically determine such variable coefficients {\it practically} requires fast transforms. Among the many possibilities, see~\cite{Hale-Townsend-15} for a new approach for a fast FFT-based discrete Legendre transform.

\section{Ultraspherical spectral method for singular integral equations}

In the following definitions, we identify $\mathbb{C}$ with $\mathbb{R}^2$ and let $\Gamma$ be bounded in $\mathbb{C}$.

\begin{definition}[Kress~\cite{Kress-10}]
A real- or complex-, scalar- or vector-valued function $f$ defined on $\Gamma$ is called {\em uniformly H\"older continuous} with {\em H\"older exponent} $0<\alpha\le1$ if there exists a constant $C$ such that
\begin{equation}
|f({\bf x})-f({\bf y})| \le C |{\bf x}-{\bf y}|^\alpha, \quad{\rm for}\quad {\bf x},{\bf y} \in \Gamma.
\end{equation}
\end{definition}
By $C^{0,\alpha}(\Gamma)$ we denote the space of all bounded and uniformly H\"older continuous functions with exponent $\alpha$. For vectors, we take $|\cdot|$ to be the Euclidean distance. With the norm
\begin{equation}
\|f\|_{0,\alpha} := \sup_{{\bf x}\in\Gamma} |f({\bf x})| + \sup_{\substack{{\bf x},{\bf y}\in\Gamma\\{\bf x}\ne{\bf y}}} \dfrac{|f({\bf x})-f({\bf y})|}{|{\bf x}-{\bf y}|^\alpha},
\end{equation}
the H\"older space is a Banach space, and we can further introduce $C^{1,\alpha}(\Gamma)$ as the space of all differentiable functions whose gradient belongs to $C^{0,\alpha}(\Gamma)$.

\begin{definition}\label{definition:Cauchy}
Let $f\in C^{0,\alpha}(\Gamma)$. The Cauchy transform over $\Gamma$ is defined as:
\begin{equation}
{\cal C}_\Gamma f(z) := \dfrac{1}{2\pi\I} \int_\Gamma \dfrac{f(\zeta)}{\zeta - z} \dkf \zeta,\quad{\rm for}\quad z\in\mathbb{C}\setminus\Gamma.
\end{equation}
\end{definition}

The Cauchy transform can be extended to $z\in\Gamma$ with integration understood as the Cauchy principal value.
\begin{definition}\label{definition:Hilbert}
Let $f\in C^{0,\alpha}(\Gamma)$. The Hilbert transform over $\Gamma$ is defined as:
\begin{equation}
{\cal H}_\Gamma f(z) := \dfrac{1}{\pi} \dashint_\Gamma \dfrac{f(\zeta)}{\zeta - z} \dkf \zeta,\quad{\rm for}\quad z\in\Gamma,
\end{equation}
where the integral is understood as the Cauchy principal value:
\begin{equation}
\dfrac{1}{\pi} \dashint_\Gamma \dfrac{f(\zeta)}{\zeta - z} \dkf \zeta = \dfrac{1}{\pi}\lim_{\rho\to0} \int_{\Gamma\setminus\Gamma(z;\rho)} \dfrac{f(\zeta)}{\zeta - z} \dkf \zeta,
\end{equation}
where $\Gamma(z;\rho) := \left\{\zeta\in\Gamma : |\zeta-z|\le\rho\right\}$.
\end{definition}

\begin{lemma}[Sokhotski--Plemelj~\cite{Sokhotski-Thesis-73,Plemelj-19-205-08}]
If $f\in C^{0,\alpha}(\Gamma)$, then:
\begin{equation}
{\cal H}_\Gamma f(z)=  \I [{\cal C}^+ + {\cal C}^-]f(z),
\end{equation}
where ${\cal C}^\pm$ denotes the limit from the left/right of $\Gamma$.
\end{lemma}

With the Hilbert and Cauchy transforms, further integrals with singularities can be defined.
\begin{definition}\label{definition:Log&Hadamard}
For $f\in C^{0}(\Gamma)$ the log transform over $\Gamma$ is defined as:
\begin{equation}
{\cal L}_\Gamma f(z) := \dfrac{1}{\pi} \int_\Gamma \log|\zeta-z|f(\zeta)\dkf \zeta,\quad{\rm for}\quad z\in\mathbb{C}.
\end{equation}
For $f\in C^{1,\alpha}(\Gamma)$ the derivative of the Hilbert transform is defined as:
\begin{equation}
{\cal H}'_\Gamma f(z) := \dfrac{1}{\pi}\ddashint_\Gamma \dfrac{f(\zeta)}{(\zeta - z)^2} \dkf \zeta,\quad{\rm for}\quad z\in\Gamma,
\end{equation}
where the integral is understood as the Hadamard finite-part~\cite{Martin-4-197-92,Monegato-50-9-94}:
\begin{equation}
\dfrac{1}{\pi}\ddashint_\Gamma \dfrac{f(\zeta)}{(\zeta - z)^2} \dkf \zeta = \dfrac{1}{\pi}\lim_{\rho\to0} \left\{\int_{\Gamma\setminus\Gamma(z;\rho)} \dfrac{f(\zeta)}{(\zeta - z)^2} \dkf \zeta - \dfrac{2f(z)}{\rho}\right\},
\end{equation}
where $\Gamma(z;\rho)$ is defined as in Definition~\ref{definition:Hilbert}.
\end{definition}

\begin{remarks}~
\begin{enumerate}
\item The Sokhotski--Plemelj lemma offers a convenient way to compute the Hilbert transform via the limit of two Cauchy transforms.
\item The use of the Cauchy principal value and the Hadamard finite-part allows for the regularization of singular and hypersingular integral operators, respectively.
\end{enumerate}
\end{remarks}

On a contour $\Gamma$, we expand the kernel of the singular integral equation~\eqref{eq:SIE} in the following way:
\begin{equation}\label{eq:SIEonI}
{\cal A} u = f, \qquad {\cal B}u = c,
\end{equation}
for 
$${\cal A} u = \dfrac{1}{\pi}\ddashint_{-1}^1 \left(\dfrac{K_1(x,y)}{(y-x)^2} + \dfrac{K_2(x,y)}{y-x} + \log|y-x| K_3(x,y) + K_4(x,y)\right) u(y){\rm\,d}y,$$
where $K_1$, $K_2$, $K_3$ and $K_4$ are known continuous bivariate kernels, $f$ is continuous, ${\cal B}$ contains $N$ linear functionals, and $u$ is the unknown solution. If in~\eqref{eq:SIEonI}, we replace the bivariate kernels with low rank approximations,
\begin{equation}\label{eq:KlambdaLowRank}
K_\lambda(x,y) \approx \sum_{i=1}^{k_{\lambda}} A_{\lambda,i}(x)B_{\lambda,i}(y),\quad{\rm for}\quad\lambda=1,2,3,4,
\end{equation}
we achieve at once two remarkable things: firstly, the approximations are compressed representations of the kernels; and secondly, the separation of variables in the low rank approximation allows for the singular integral operators to be constructed via the Definitions~\ref{definition:Hilbert} and~\ref{definition:Log&Hadamard}.

In the following two subsections, we consider the case where $\Gamma$  is the unit interval, and emulate the construction of the ultraspherical spectral method for ODEs to arrive at an almost-banded system to represent~\eqref{eq:SIEonI}.  In this setting, we must use weighted Chebyshev bases to accomplish this task. Note that alternative spectral methods for open arcs are discussed in~\cite{Lintner-Bruno-12}.

\subsection{Inverse square root endpoint singularities}

Indeed, the Hilbert transform of weighted Chebyshev polynomials is known~\cite{King-1-09}:
\begin{equation}
{\cal H}_{(-1,1)}\left[\dfrac{T_n(x)}{\sqrt{1-x^2}}\right] = \left\{\begin{array}{cc}0, & n=0,\\C^{(1)}_{n-1}(x), & n\ge1,\end{array}\right.
\end{equation}
This operation can then be expressed as the banded operator from the weighted Chebyshev coefficients to the ultraspherical coefficients of order $1$:
\begin{equation}
{\cal H}_{(-1,1)} = \begin{pmatrix} 0 & 1\\
&&1\\
&&&1\\
&&&&\ddots\end{pmatrix}.
\end{equation}
Upon integration with respect to $x$, we obtain an expression for the log transform:
\begin{equation}
{\cal L}_{(-1,1)}\left[\dfrac{T_n(x)}{\sqrt{1-x^2}}\right] = \left\{\begin{array}{cc}-\log2, & n=0,\\-\dfrac{T_{n}(x)}{n}, & n\ge1,\end{array}\right.
\end{equation}
or as an operator from the weighted Chebyshev coefficients to the Chebyshev coefficients:
\begin{equation}
{\cal L}_{(-1,1)} =\begin{pmatrix}
-\log 2\\
& -1\\
& & -\frac{1}{2}\\
& & & \ddots\\
\end{pmatrix}.
\end{equation}
In addition, upon differentiation with respect to $x$, we also obtain an expression for the derivative of the Hilbert transform:
\begin{equation}
{\cal H}'_{(-1,1)}\left[\dfrac{T_n(x)}{\sqrt{1-x^2}}\right] = \left\{\begin{array}{cc}0, & n=0,1,\\C^{(2)}_{n-2}(x), & n\ge2,\end{array}\right.
\end{equation}
This operation can then be expressed as the banded operator from the weighted Chebyshev coefficients to the ultraspherical coefficients of order $2$:
\begin{equation}
{\cal H}'_{(-1,1)} = \begin{pmatrix} 0 & 0 & 1\\
&&&1\\
&&&&1\\
&&&&&\ddots\end{pmatrix}.
\end{equation}

Lastly, the orthogonality of the Chebyshev polynomials immediately yields for the functional
\begin{equation}
\Sigma_\Gamma f := \dfrac{1}{\pi}\int_\Gamma f(\zeta){\rm\,d}\zeta
\end{equation}
the following:
\begin{equation}
\Sigma_{(-1,1)}\left[\dfrac{T_n(x)}{\sqrt{1-x^2}}\right] = \left\{\begin{array}{cc}1, & n=0,\\0, & n\ge1,\end{array}\right.
\end{equation}
or as a compact functional on the weighted Chebyshev coefficients:
\begin{equation}
\Sigma_{(-1,1)} =\begin{pmatrix}
1 & 0 & 0 &\cdots\\
\end{pmatrix}.
\end{equation}

Combining the integral operators together with the bivariate approximations, we define:
\begin{align}
{\cal H}'_{(-1,1)}[K_1] & := \sum_{i=1}^{k_1} {\cal M}_2[A_{1,i}(x)]{\cal H}'_{(-1,1)} {\cal M}_0[B_{1,i}(y)],\\
{\cal H}_{(-1,1)}[K_2] & := \sum_{i=1}^{k_2} {\cal M}_1[A_{2,i}(x)]{\cal H}_{(-1,1)} {\cal M}_0[B_{2,i}(y)],\\
{\cal L}_{(-1,1)}[K_3] & := \sum_{i=1}^{k_3} {\cal M}_0[A_{3,i}(x)]{\cal L}_{(-1,1)} {\cal M}_0[B_{3,i}(y)],\\
\Sigma_{(-1,1)}[K_4] & := \sum_{i=1}^{k_4} {\cal M}_0[A_{4,i}(x)]\Sigma_{(-1,1)} {\cal M}_0[B_{4,i}(y)].
\end{align}
Then, we can reduce singular integral equations of the form~\eqref{eq:SIEonI} into an infinite-dimensional almost-banded system:
\begin{align}
{\cal B}{\bf u} & = {\bf c},\nonumber\\
\left({\cal H}'_{(-1,1)}[K_1] + {\cal S}_1{\cal H}_{(-1,1)}[K_2] + {\cal S}_1{\cal S}_{0}({\cal L}_{(-1,1)}[K_3] + \Sigma_{(-1,1)}[K_4])\right) {\bf u} & = {\cal S}_1{\cal S}_0{\bf f}.
\end{align}
This system can be solved directly using the framework of infinite-dimensional linear algebra~\cite{Olver-Townsend-57-14}, built out of the adaptive QR factorization introduced in \cite{Olver-Townsend-55-462-13}.  

\subsection{Square root endpoint singularities}

The Hilbert transform of weighted Chebyshev polynomials of the second kind is also known~\cite{King-1-09}:
\begin{equation}
{\cal H}_{\mathbb{I}}\left[U_{n}(x)\sqrt{1-x^2}\right] = -T_{n+1}(x),\quad n\ge0.
\end{equation}
This operation can then be expressed as the banded operator from the weighted ultraspherical coefficients of order $1$ to the Chebyshev coefficients:
\begin{equation}
{\cal H}_{\mathbb{I}} = \begin{pmatrix} 0\\
-1\\
&-1\\
&&\ddots\end{pmatrix}.
\end{equation}
Upon integration with respect to $x$, we obtain an expression for the log transform:
\begin{equation}
{\cal L}_{\mathbb{I}}\left[U_n(x)\sqrt{1-x^2}\right] = \left\{\begin{array}{cc}-\dfrac{1}{2}\log2+\dfrac{1}{4}T_2(x), & n=0,\\&\\ \dfrac{1}{2}\left(\dfrac{T_{n+2}(x)}{n+2}-\dfrac{T_{n}(x)}{n}\right), & n\ge1,\end{array}\right.
\end{equation}
or as an operator from the weighted ultraspherical coefficients of order $1$ to the Chebyshev coefficients:
\begin{equation}
{\cal L}_{\mathbb{I}} =\begin{pmatrix}
-\frac{1}{2}\log 2\\
0& -\frac{1}{2}\\
\frac{1}{4}& 0 & -\frac{1}{4}\\
& \ddots& \ddots& \ddots\\
\end{pmatrix}.
\end{equation}
In addition, upon differentiation with respect to $x$, we also obtain an expression for the derivative of the Hilbert transform:
\begin{equation}
{\cal H}'_{\mathbb{I}}\left[U_n(x)\sqrt{1-x^2}\right] = -(n+1)C^{(1)}_n(x),\quad n\ge0,
\end{equation}
This operation can then be expressed as the banded operator from the weighted ultraspherical coefficients of order $1$ to the ultraspherical coefficients of order $1$:
\begin{equation}
{\cal H}'_{\mathbb{I}} = \begin{pmatrix} -1\\
&-2\\
&&-3\\
&&&\ddots\end{pmatrix}.
\end{equation}

Lastly, the orthogonality of the Chebyshev polynomials of the second kind immediately yields for $\Sigma_{\mathbb{I}}$:
\begin{equation}
\Sigma_{\mathbb{I}}\left[U_n(x)\sqrt{1-x^2}\right] = \left\{\begin{array}{cc}\frac{1}{2}, & n=0,\\0, & n\ge1,\end{array}\right.
\end{equation}
or as a compact functional on the weighted Chebyshev basis:
\begin{equation}
\Sigma_{\mathbb{I}} =\begin{pmatrix}
\frac{1}{2} & 0 & 0 &\cdots\\
\end{pmatrix}.
\end{equation}

Combining the integral operators together with the bivariate approximations, we define:
\begin{align}
{\cal H}'_{\mathbb{I}}[K_1] & := \sum_{i=1}^{k_1} {\cal M}_1[A_{1,i}(x)]{\cal H}'_{\mathbb{I}} {\cal M}_1[B_{1,i}(y)],\\
{\cal H}_{\mathbb{I}}[K_2] & := \sum_{i=1}^{k_2} {\cal M}_0[A_{2,i}(x)]{\cal H}_{\mathbb{I}} {\cal M}_1[B_{2,i}(y)],\\
{\cal L}_{\mathbb{I}}[K_3] & := \sum_{i=1}^{k_3} {\cal M}_0[A_{3,i}(x)]{\cal L}_{\mathbb{I}} {\cal M}_1[B_{3,i}(y)],\\
\Sigma_{\mathbb{I}}[K_4] & := \sum_{i=1}^{k_4} {\cal M}_0[A_{4,i}(x)]\Sigma_{\mathbb{I}} {\cal M}_1[B_{4,i}(y)],
\end{align}
and in the framework of infinite-dimensional linear algebra~\cite{Olver-Townsend-57-14}, we may solve singular integral equations of the form~\eqref{eq:SIEonI} via the almost-banded system:
\begin{align}
{\cal B}{\bf u} & = {\bf c},\nonumber\\
\left({\cal H}'_{\mathbb{I}}[K_1] + {\cal S}_{0}({\cal H}_{\mathbb{I}}[K_2] +  {\cal L}_{\mathbb{I}}[K_3] + \Sigma_{\mathbb{I}}[K_4])\right) {\bf u} & = {\cal S}_0{\bf f}.
\end{align}

Let $m_x+m_y$ denote the largest sum of degrees of the bivariate Chebyshev expansions of the integral kernels such that for some $\epsilon>0$:
\begin{equation}
\left\|K_\lambda(x,y) - \sum_{i=1}^{k_{\lambda}} A_{\lambda,i}(x)B_{\lambda,i}(y)\right\|_\infty \le \epsilon\|K_\lambda(x,y)\|_\infty,\quad{\rm for}\quad\lambda=1,2,3,4.
\end{equation}
Then, the complexity of the adaptive QR factorization is ${\cal O}((m_x+m_y)^2n)$ operations, where $n$ is degree of the resulting weighted Chebyshev expansion of the solution.  This is reduced to ${\cal O}((m_x+m_y)n)$ operations  by pre-caching the QR factorization.

\begin{remarks}~
\begin{enumerate}
\item The observation that $|{\rm d}\zeta| = {\rm d}\zeta$ on $\mathbb{I}$ allows us to relate line integral formulations with the operators of Definitions~\ref{definition:Hilbert} and \ref{definition:Log&Hadamard}\footnote{Both variants of the singular integral operators are implemented in {\tt SingularIntegralEquations.jl}.}.
\item Mixed equations involving derivatives and singular integral operators are also covered in this framework.
\item It is straightforward to obtain the singular integral operators on arbitrary (complex) intervals $(a,b)$ using an affine map.
\end{enumerate}
\end{remarks}

\subsection{Multiple disjoint contours}

Singular integral equations on a union of disjoint intervals $\Gamma=\Gamma_1\cup\Gamma_2\cup\cdots\cup\Gamma_d$ are covered in this framework.   We can decompose~\eqref{eq:SIEonI} as
\begin{equation}\label{eq:Au=fsystem}
\begin{pmatrix}
{\cal B}_1 & {\cal B}_2 & \cdots & {\cal B}_d\\
{\cal A}_{1,1} & {\cal A}_{1,2} & \cdots & {\cal A}_{1,d}\\
{\cal A}_{2,1} & {\cal A}_{2,2} & \cdots & {\cal A}_{2,d}\\
\vdots & \vdots & \ddots & \vdots\\
{\cal A}_{d,1} & {\cal A}_{d,2} & \cdots & {\cal A}_{d,d}\\
\end{pmatrix}
\begin{pmatrix}
{\bf u}_1\\{\bf u}_2\\\vdots\\{\bf u}_d
\end{pmatrix}
 =
\begin{pmatrix}
{\bf c}\\{\bf f}_1\\{\bf f}_2\\\vdots\\{\bf f}_d
\end{pmatrix},
\end{equation}
where  each ${\cal B}_i$ is a set of linear functionals and ${\cal A}_{i,j} = {\cal A}_{\Gamma_i}|_{\Gamma_j}$.    The diagonal blocks are equivalent to the previous case considered, hence result in banded representations.  The off-diagonal blocks can be constructed directly by expanding the entire non-singular kernel in low rank form and using the compact functionals $\Sigma_{(-1,1)}$ or $\Sigma_{\mathbb{I}}$.  The resulting representation is, in fact, finite-dimensional and hence every block is banded.  

Here, we show how a block-almost-banded infinite-dimensional system can be interlaced to be re-written as a single infinite-dimensional and almost-banded system.   Re-ordering both vectors $({\bf u}_1,{\bf u}_2,\ldots,{\bf u}_d)^\top$ and $({\bf f}_1,{\bf f}_2,\ldots,{\bf f}_d)^\top$ to:
\begin{align}
{\bf U}& =
\begin{pmatrix}
u_{1,0}&u_{2,0}&\cdots&u_{d,0}&u_{1,1}&u_{2,1}&\cdots&u_{d,1}&\cdots
\end{pmatrix}^\top\\
{\bf F}& =
\begin{pmatrix}
f_{1,0}&f_{2,0}&\cdots&f_{d,0}&f_{1,1}&f_{2,1}&\cdots&f_{d,1}&\cdots
\end{pmatrix}^\top
\end{align}
amounts to a permutation of almost every row and column in~\eqref{eq:Au=fsystem}. Define each entry of $\mathfrak{B}$ and $\mathfrak{A}$ by:
\begin{align}
\mathfrak{B}_{i,j} & = {\cal B}_{\{(i-1){\rm\,mod\,}d\}+1,\lfloor\frac{i+d-1}{d}\rfloor,j},\\
\mathfrak{A}_{i,j} & = {\cal A}_{\{(i-1){\rm\,mod\,}d\}+1,\{(j-1){\rm\,mod\,}d\}+1,\lfloor\frac{i+d-1}{d}\rfloor,\lfloor\frac{j+d-1}{d}\rfloor},
\end{align}
where the last two indices in each term on the right-hand sides denote the entries of the functional or operator.
This perfect shuffle allows for the system~\eqref{eq:Au=fsystem} to be re-written as the almost-banded system
\begin{equation}
\begin{pmatrix}
\mathfrak{B}\\\mathfrak{A}
\end{pmatrix}{\bf U} = \begin{pmatrix}{\bf c}\\{\bf F}\end{pmatrix}.
\end{equation}

\subsection{Diagonal preconditioners for compactness}

We now show that our formulations leads to equations whose operators are compact perturbations of the identity. For well-posed (integral) equations, this ensures convergence~\cite{Olver-Townsend-55-462-13}. We show this for the singular operators in equations~\eqref{eq:DirichletIE} and~\eqref{eq:NeumannIE} defined on the canonical unit interval and in suitably chosen spaces. Note that a similar analysis is performed in~\cite{Lintner-Bruno-12}. Since we are working in coefficient space, we consider the problem as defined in $\ell_\lambda^2$ spaces. In the case of Chebyshev expansions, this corresponds to Sobolev spaces of the transformed function $u(\cos\theta)$.
\begin{definition}[Olver and Townsend~\cite{Olver-Townsend-55-462-13}]\label{definition:discreteSobolev}
The space $\ell_\lambda^2\subset\mathbb{C}^\infty$ is defined as the Banach space with norm:
\begin{equation}
\|{\bf u}\|_{\ell_\lambda^2} = \sqrt{\sum_{k=0}^\infty|u_k|^2(k+1)^{2\lambda}} <\infty.
\end{equation}
\end{definition}

Let ${\cal P}_n = (I_n,{\bf 0})$ be the projection operator.

\begin{lemma}\label{lemma:DirichletI+K}
For the Dirichlet problem singular integral operator in~\eqref{eq:DirichletIE}, if $\Phi$ takes the form~\eqref{eq:GFisAlogplusB} with $A$ and $B$ analytic in both $x$ and $y$ and if we take ${\cal R}$ to be
\begin{equation}
{\cal R} = 2\begin{pmatrix}
\frac{1}{\log 2}\\
& 1\\
& & 2\\
& & & 3\\
& & & & \ddots
\end{pmatrix} : \ell_{\lambda}^2\to\ell_{\lambda-1}^2,
\end{equation}
then
\begin{equation}
\begin{pmatrix}{\cal L}_{(-1,1)}[\pi A] + \Sigma_{(-1,1)}[\pi B]\end{pmatrix}{\cal R} = I + {\cal K},
\end{equation}
where ${\cal K} : \ell_\lambda^2\to\ell_\lambda^2$ is compact for $\lambda\in\mathbb{R}$.
\end{lemma}
\begin{proof}
Since $A(x,x) = -(2\pi)^{-1}$, we let $\tilde{A}(x,y) \equiv A(x,y) - A(x,x)$ and separate the operator~\eqref{eq:DirichletIE} as:
\begin{equation}
{\cal L}_{(-1,1)}[\pi A(x,x)] + {\cal L}_{(-1,1)}[\pi \tilde{A}(x,y)] + \Sigma_{(-1,1)}[\pi B].
\end{equation}
It is straightforward to show
\begin{equation}
{\cal L}_{(-1,1)}[\pi A(x,x)]{\cal R} = I : \ell_\lambda^2\to\ell_\lambda^2.
\end{equation}
Then, we need to show that the remainder is compact. Since:
\begin{equation}
\|{\cal P}_n{\cal L}_{(-1,1)}{\cal P}_n^\top - {\cal L}_{(-1,1)}\|\to0\quad{\rm as}\quad n\to\infty,
\end{equation}
${\cal L}_{(-1,1)} : \ell_\lambda^2 \to\ell_\lambda^2 $ is compact. Compactness of $\Sigma_{(-1,1)}$ is implied by its finite-rank.
Expanding $\tilde{A}$ and $B$ in low rank Chebyshev approximants, we have:
\begin{equation}
\pi\left(\sum_{i=1}^{k_{\tilde{A}}} {\cal M}_0[\tilde{A}_{1,i}(x)] {\cal L}_{(-1,1)} {\cal M}_0[\tilde{A}_{2,i}(y)] + \sum_{i=1}^{k_B} {\cal M}_0[B_{1,i}(x)] \Sigma_{(-1,1)} {\cal M}_0[B_{2,i}(y)]\right) {\cal R}.
\end{equation}
Since $A$ and $B$ are analytic with respect to $y$, then for every $i$ and for every $\lambda\in\mathbb{R}$:
\begin{subequations}
\begin{align}
{\cal M}_0[\tilde{A}_{2,i}(y)] : \ell_{\lambda-1}^2 \to \ell_\lambda^2 \quad &\Longrightarrow \quad {\cal M}_0[\tilde{A}_{2,i}(y)] {\cal R} : \ell_\lambda^2 \to \ell_\lambda^2,\\
{\cal M}_0[B_{2,i}(y)] : \ell_{\lambda-1}^2 \to \ell_\lambda^2 \quad &\Longrightarrow \quad {\cal M}_0[B_{2,i}(y)] {\cal R} : \ell_\lambda^2 \to \ell_\lambda^2,
\end{align}
\end{subequations}
are bounded. Compactness follows from the linear combination of a product of bounded and compact operators being compact.
\end{proof}

\begin{lemma}\label{lemma:NeumannI+K}
For the Neumann problem singular integral operator in~\eqref{eq:NeumannIE}, if $\Phi$ takes the form~\eqref{eq:GFisAlogplusB} with $A$ and $B$ analytic in both $x$ and $y$ and if we take ${\cal R}$ to be
\begin{equation}
{\cal R} = -2\begin{pmatrix}
1\\
& \frac{1}{2}\\
& & \frac{1}{3}\\
& & & \frac{1}{4}\\
& & & & \ddots
\end{pmatrix} : \ell_{\lambda}^2\to\ell_{\lambda+1}^2,
\end{equation}
then:
\begin{equation}
\begin{pmatrix}{\cal H}'_{\mathbb{I}}[-\pi A] + {\cal L}_{\mathbb{I}}[\pi A''] + \Sigma_{\mathbb{I}}[\pi B'']\end{pmatrix}{\cal R} = I + {\cal K},
\end{equation}
where the two primes indicate:
\begin{equation}
A''(x,y) = \left.\dfrac{\partial^2 A({\bf x},{\bf y})}{\partial x_2\partial y_2}\right|_{{\bf x},{\bf y} = (x,0),(y,0)},
\end{equation}
and where ${\cal K} : \ell_\lambda^2\to\ell_\lambda^2$ is compact for $\lambda\in\mathbb{R}$.
\end{lemma}
\begin{proof}
Since $A(x,x) = -(2\pi)^{-1}$, we let $\tilde{A}(x,y) \equiv A(x,y) - A(x,x)$ and separate the operator~\eqref{eq:NeumannIE} as:
\begin{equation}
{\cal H}'_{\mathbb{I}}[-\pi A(x,x)] + {\cal H}'_{\mathbb{I}}[-\pi\tilde{A}(x,y)] + {\cal L}_{\mathbb{I}}[\pi A''] + \Sigma_{\mathbb{I}}[\pi B''].
\end{equation}
It is straightforward to show:
\begin{equation}
{\cal H}'_{\mathbb{I}}[-\pi A(x,x)]{\cal R} = I : \ell_\lambda^2\to\ell_\lambda^2.
\end{equation}
Then, we need to show that the remainder is compact. Since:
\begin{equation}
\|{\cal P}_n{\cal R}{\cal P}_n^\top - {\cal R}\|\to0\quad{\rm as}\quad n\to\infty,
\end{equation}
${\cal R} : \ell_\lambda^2 \to\ell_\lambda^2 $ is compact. Furthermore, showing boundedness of ${\cal S}_0, {\cal L}_{\mathbb{I}}, \Sigma_{\mathbb{I}}:\ell_\lambda^2\to\ell_\lambda^2$ and ${\cal H}'_{\mathbb{I}}:\ell_{\lambda+1}^2\to\ell_\lambda^2$ is straightforward. Expanding $\tilde{A}$, $A''$ and $B''$ in low rank Chebyshev and ultraspherical approximants, we have:
\begin{align}
&\pi\left(-\sum_{i=1}^{k_{\tilde{A}}} {\cal M}_1[\tilde{A}_{1,i}(x)] {\cal H}'_{\mathbb{I}} {\cal M}_1[\tilde{A}_{2,i}(y)]\right.\nonumber\\
& + \left.{\cal S}_0\left(\sum_{i=1}^{k_{A''}} {\cal M}_0[A''_{1,i}(x)] {\cal L}_{\mathbb{I}} {\cal M}_1[A''_{2,i}(y)] + \sum_{i=1}^{k_{B''}} {\cal M}_0[B''_{1,i}(x)] \Sigma_{\mathbb{I}} {\cal M}_1[B''_{2,i}(y)]\right)\right) {\cal R}.
\end{align}
Since $A$ and $B$ are analytic with respect to $y$, then for every $i$ and for every $\lambda\in\mathbb{R}$:
\begin{equation}
{\cal M}_1[\tilde{A}_{2,i}(y)] : \ell_{\lambda}^2 \to \ell_{\lambda+1}^2 \quad \Longrightarrow \quad {\cal H}'_{\mathbb{I}}{\cal M}_1[\tilde{A}_{2,i}(y)] : \ell_\lambda^2 \to \ell_\lambda^2,
\end{equation}
are bounded. Compactness follows from the linear combination of a product of bounded and compact operators being compact.
\end{proof}

\begin{remarks}~
\begin{enumerate}
\item For complicated fundamental solutions whose bivariate low rank Chebyshev approximants have large degrees, preconditioners such as those in Lemmas~\ref{lemma:DirichletI+K} and~\ref{lemma:NeumannI+K} allow for continuous Krylov subspace methods or conjugate gradients on the normal equations to converge in a relatively fewer number of iterations compared with the un-preconditioned operators.   Furthermore, the low rank Chebyshev approximants allow for the operator-function product to be carried out in ${\cal O}((m+n)\log(m+n))$, where $m$ is the largest degree of a multiplication operator and $n$ is the degree of the Chebyshev approximant of the solution.  Iterative solvers are outside the scope of this article, however.
\item Operator preconditioners~\cite{Hiptmair-52-699-06} can also be derived which would yield similar $I+{\cal K}$ results. However, working in coefficient space allows for a simpler exposition.
\end{enumerate}
\end{remarks}

\subsection{Numerical evaluation of Cauchy and log transforms on intervals}~\label{subsection:logkernel} Fast and spectrally accurate numerical evaluation of the scattered far-field can be derived from Clenshaw--Curtis integration of the fundamental solution multiplied by the density. For each evaluation point, the fundamental solution can be sampled at the $2N$ roots of the $2N^{\rm th}$ degree Chebyshev polynomial, where $N$ is the length of the polynomial representation of the density. Since the resulting density may be as complicated\footnote{In the Helmholtz equation, for example, both the density and the fundamental solution are oscillatory with the same wavenumber.} as the fundamental solution itself, doubling the length is sufficient to resolve the coefficients of the fundamental solution multiplied by the density.

It is well known that such an evaluation technique is inaccurate near the boundary~\cite{Barnett-36-A427-14}. In the context of Riemann--Hilbert problems, spectrally accurate evaluation {\it near and far from} $\Gamma$ can be obtained by exact integration of a modified Chebyshev series that encodes vanishing conditions at the endpoints.

Consider the modified Chebyshev series:
\begin{equation}
\hat{T}_0(x) := 1,\quad \hat{T}_1(x) := x, \quad \hat{T}_n(x) := T_n(x) - T_{n-2}(x),\qquad n\ge2.
\end{equation}
If we expand $u$ in a Chebyshev series and this modified Chebyshev series:
\begin{equation}
u(x) = \sum_{n=0}^\infty u_n T_n(x) = \sum_{n=0}^\infty \hat{u}_n\hat{T}_n(x),
\end{equation}
then we have the relation:
\begin{equation}
\begin{pmatrix}
u_0\\u_1\\u_2\\\vdots\\
\end{pmatrix}
=
\begin{pmatrix}
1 & 0 & -1\\
& 1 & 0 & -1\\
& & 1 & 0 & -1\\
& & & \ddots & \ddots & \ddots\\
\end{pmatrix}
\begin{pmatrix}
\hat{u}_0\\\hat{u}_1\\\hat{u}_2\\\vdots\\
\end{pmatrix}.
\end{equation}
Therefore, any finite sequences $\{u_n\}_{n=0}^N$ and $\{\hat{u}_n\}_{n=0}^N$ can be transformed to the other in ${\cal O}(N)$ operations, either via forward application of the banded operator, or via an in-place back substitution.

Lemmas \ref{Lemma:CauchyUkBasis} and \ref{Lemma:InvSqrtCauchy} contain formul\ae\ for Cauchy transforms of weighted Chebyshev polynomials evaluated in the complex plane.  These were originally derived in this form in~\cite{Trogdon-Olver-15}, based on results in~\cite{Olver-11-153-11,Olver-80-1745-11,Olver-163-1185-11,Olver-Trogdon-39-101-13}.  These formul\ae~are adapted in Lemma~\ref{Lemma:LogKernelFormula} for the log transform as well.

\begin{definition}
Define the Joukowsky transform:
\begin{equation}
J(z) := \dfrac{z + z^{-1}}{2},
\end{equation}
and one of its inverses:
\begin{equation}
\Jin(z) := z - \sqrt{z-1} \sqrt{z+1},
\end{equation}
which maps the slit plane $\mathbb{C}\setminus\mathbb{I}$ to the unit disk.
\end{definition}
The Joukowsky transform is useful for proving and summarizing the following results.

\begin{lemma}[Lemma 5.6~\cite{Trogdon-Olver-15}]\label{Lemma:CauchyUkBasis} For $k \geq 0$:
\begin{equation}
\CC_{\mathbb I}[\sqrt{1 - \diamond^2} U_k](z) = \dfrac{\I}{2}\Jin(z)^{k + 1}.
\end{equation}
\end{lemma}
\begin{proof}
We verify that the Sokhotski-Plemelj lemma is satisfied. Note that for $x =\cos \theta$ we have:
\begin{equation}
\lim_{\epsilon\searrow 0}\Jin(x\pm \I \epsilon) =  x \mp \I \sqrt{1-x^2}= \cos \theta \mp \I \sin \theta = \E^{\mp\I \theta}.
\end{equation}
It follows that:
\begin{align}
\lim_{\epsilon\searrow 0}\dfrac{\Jin(\cos\theta+ \I \epsilon)^{k + 1} - \Jin(\cos\theta- \I \epsilon)^{k + 1} }{ 2 \I}& = \dfrac{\E^{-\I (k+1) \theta} - \E^{\I (k+1) \theta} }{ 2 \I},\nonumber\\
= -\sin (k+1) \theta & = -U_k(\cos \theta) \sqrt{1-\cos^2 \theta}.
\end{align}
\end{proof}

\begin{lemma}[Lemma 5.11~\cite{Trogdon-Olver-15}]\label{Lemma:InvSqrtCauchy} For $k \geq 2$:
\begin{align}
\CC_{(-1,1)}\br[\dfrac{1 }{ \sqrt{1 - \diamond^2}}]\!(z) & = \dfrac{\I }{ 2 \sqrt{z-1}\sqrt{z+1}},\\
\CC_{(-1,1)}\br[\dfrac{\diamond }{ \sqrt{1 - \diamond^2}}]\!(z) & = \dfrac{\I z }{ 2\sqrt{z-1}\sqrt{z+1}} - \dfrac{\I }{ 2},\quad{\rm and}\\
\CC_{(-1,1)}\br[\dfrac{\hat{T}_k}{ \sqrt{1 - \diamond^2}}]\!(z) & = -\I\, \Jin(z)^{k-1}.
\end{align}
\end{lemma}	
\begin{proof}
The first two parts follow immediately from the Sokhotski-Plemelj lemma. The last part follows since:
\begin{equation}
\sin (k-1) \theta = \dfrac{\cos k \theta - \cos (k-2) \theta }{ 2 \sin \theta}.
\end{equation}
\end{proof}

We extend these results here to the log transform.

\begin{lemma}\label{Lemma:LogKernelFormula}
\begin{align}
		{\cal L}_{\mathbb{I}} \br[{ \sqrt{1-\diamond^2}}](z)&=\Re \dfrac{\Jin(z)^2 }{ 4} - \dfrac{\log \abs{\Jin(z)} +\log 2 }{ 2},\\	
		{\cal L}_{\mathbb{I}} \br[{U_k \sqrt{1-\diamond^2}}](z)&=\half \Re \br[ \dfrac{\Jin(z)^{k+2} }{ k+2} - \dfrac{\Jin(z)^{k} }{ k}],\\
		{\cal L}_{(-1,1)} \br[\dfrac{1 }{ \sqrt{1 - \diamond^2}}](z)&=-\log\abs{\Jin(z)} - \log 2,\\
		{\cal L}_{(-1,1)} \br[\dfrac{\diamond }{ \sqrt{1 - \diamond^2}}](z)&=-\Re\Jin(z),\\		
		{\cal L}_{(-1,1)} \br[\dfrac{\hat{T}_2 }{ \sqrt{1 - \diamond^2}}](z)&=\log\abs{\Jin(z)} + \log 2 - \Re \dfrac{\Jin(z)^2}{ 2},\quad{\rm and}\\
		{\cal L}_{(-1,1)} \br[\dfrac{\hat{T}_k }{ \sqrt{1 - \diamond^2}}](z)&=\Re \br[ \dfrac{\Jin(z)^{k-2} }{ k-2} - \dfrac{\Jin(z)^{k} }{ k}].
\end{align}
\end{lemma}
\begin{proof}
These formul\ae\ follow from integrating the formul\ae\ for Cauchy transforms and taking the real part.  We can compute the indefinite integrals directly~\cite[\S 5.4.4]{Trogdon-Olver-15}:
\begin{align}
		\int^z  \dfrac{1 }{  \sqrt{z - 1} \sqrt{z + 1}} \dkf z &=- \log \Jin(z),\\
		\int^z  \dfrac{1-z }{  \sqrt{z - 1} \sqrt{z + 1}}  \dkf z &= \Jin(z),\\
		2\int^z  \Jin(z)  \dkf z &=  \dfrac{\Jin(z)^2 }{ 2} - \log\, \Jin(z),\quad{\rm and}\\
		2\int^z \Jin(z)^k \dkf z &= \dfrac{\Jin(z)^{k+1} }{ k+1} - \dfrac{\Jin(z)^{k-1} }{ k-1},\quad{\rm for}\quad k\geq 2.
\end{align}
We also have the normalization for $z\rightarrow +\infty$:
\begin{equation}
\int_{-1}^1 f(x) \log(z-x) \dkf x = \log z \int_{-1}^1 f(x) \dkf x + {\cal O}(z^{-1}).
\end{equation}
Note that:
\begin{equation}
\Jin(z) \sim \dfrac{1}{2z} + {\cal O}(z^{-3})\quad{\rm as}\quad z \rightarrow \infty,
\end{equation}
hence:
\begin{equation}
\log\Jin(z) = -\log z-\log 2 +{\cal O}(z^{-1})\quad{\rm as}\quad z\rightarrow \infty.
\end{equation}
\end{proof}

These formul\ae~can be generalized to other intervals, including in the complex plane, by using a straightforward change of variables:
\begin{align}
	{\cal L}_{(a,b)}f(z) 
	= \dfrac{\abs{b-a} }{ 2}& {\cal L}_{(-1,1)}\br[f\!\pr({\dfrac{b+a }{ 2} + \dfrac{b-a }{ 2}\diamond})]\pr(\dfrac{b+a-2z }{ b-a})\nonumber\\
	&\qquad+ \dfrac{\abs{b-a}  }{ 2 \pi} \log \dfrac{\abs{b-a} }{ 2} \int_{-1}^1 f\!\pr({\dfrac{b+a }{ 2} + \dfrac{b-a }{ 2}x}) \dkf x.
\end{align}

\section{Applications}\label{section:applications}

\subsection{The Faraday cage}\label{subsection:faraday}

The Faraday cage effect describes how a wire mesh can reduce the strength of the electric field within its confinement. This phenomenon was described as early as 1755 by Franklin~\cite[\S 2-18]{Kraus-92} and in 1836 by Faraday~\cite{Faraday-39}. While the description of the phenomenon is quite prevalent in undergraduate material on electrostatics, a standard mathematical analysis has been missing until only recently by Martin~\cite{Martin-470-20140344-14} and Chapman, Hewett and Trefethen~\cite{Chapman-Hewett-Trefethen-57-398-15}. In~\cite{Chapman-Hewett-Trefethen-57-398-15}, three different approaches are considered for numerical simulations: a collocated least squares direct numerical calculation, a homogenized approximation via coupling of the solutions at multiple scales, and an approximation by point charges determined by minimizing a quadratic energy functional.

In~\cite{Chapman-Hewett-Trefethen-57-398-15}, it is shown that the shielding of a Faraday cage of circular wires centred at the roots of unity is a linear phenomenon instead of providing exponential shielding as the number of wires tends to infinity for geometrically feasible radii, i.e.~radii that prevent overlapping. In their synopsis, it is claimed that a Faraday cage with any arbitrarily shaped objects will not provide considerably different shielding in the asymptotic limit. Here, we confirm this observation with infinitesimally thin plates of the same electrostatic capacity as wires\footnote{This corresponds to plates of width $4r$ where $r$ is the wire radius.} angled normal to the vector from the origin to their centres. Our numerical results are in excellent asymptotic agreement with those presented in~\cite{Chapman-Hewett-Trefethen-57-398-15}. Departing from the practical case of normal plates, we also consider infinitesimally thin plates angled tangential to the vector from the origin to their centres. In this case, we escape the practical material limit on the number of shields as an infinite number of plates can be modelled independent of radial parameter.

We seek to find the solution to the Laplace equation such that, in addition:
\begin{subequations}
\begin{align}
\Delta u({\bf x}) & = 0, &{\rm for}\quad{\bf x}\in \Omega,\\
u({\bf x}) & = u_0, &{\rm for}\quad {\bf x}\in\Gamma,\label{eq:FaradayCondition1}\\
u({\bf x}) & = \log| {\bf x} - {\bf y}| + {\cal O}(1), &{\rm as}\quad |{\bf x}- {\bf y}|\to 0,\\
u({\bf x}) & = \log| {\bf x} | + o(1), &{\rm as}\quad |{\bf x}|\to\infty.
\end{align}
\end{subequations}

Since this is a Dirichlet problem, we begin by splitting the solution $u = u^i + u^s$, where:
\begin{equation}
u^i({\bf x}) = \log|{\bf x}-{\bf y}| = 2\pi\Phi({\bf x},{\bf y}),
\end{equation}
is the source term with strength $2\pi$ located at ${\bf y} = (2,0)$, as in~\cite{Chapman-Hewett-Trefethen-57-398-15}. We represent $u^s$ in terms of a density with the single-layer potential equal to the effect of the logarithmic source. Alone, this represents a solution to the Laplace equation with Dirichlet boundary conditions on $\Gamma$. To satisfy condition~\eqref{eq:FaradayCondition1}, we augment our system to ensure there is a constant charge of zero on the wires and plates:
\begin{equation}
\int_\Gamma \left[\dfrac{\partial u^s}{\partial n}\right]{\rm\,d}\Gamma({\bf y}) = 0,
\end{equation}
though each wire may individually carry a different charge, and the unknown constant $u_0$ to accommodate this condition. Figure~\ref{fig:Laplace1} shows the numerical results for shielding by normal and tangential plates. Figure~\ref{fig:Laplace2} shows a plot of the convergence of the density coefficients and the field strength at the origin for various parameter values.

\begin{figure}[htbp]
\begin{center}
\begin{tabular}{cc}
\includegraphics[width=0.485\textwidth]{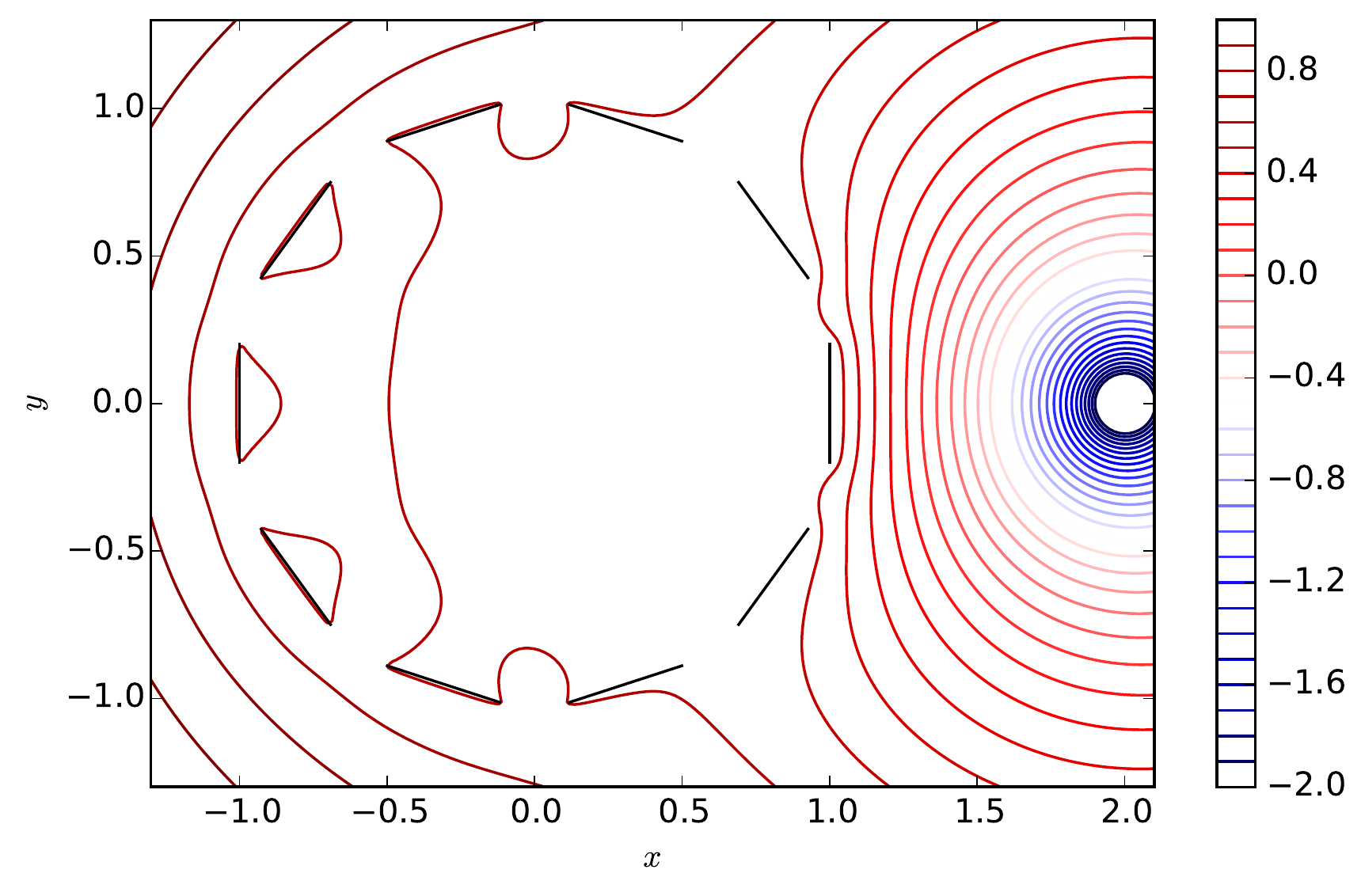}&
\includegraphics[width=0.485\textwidth]{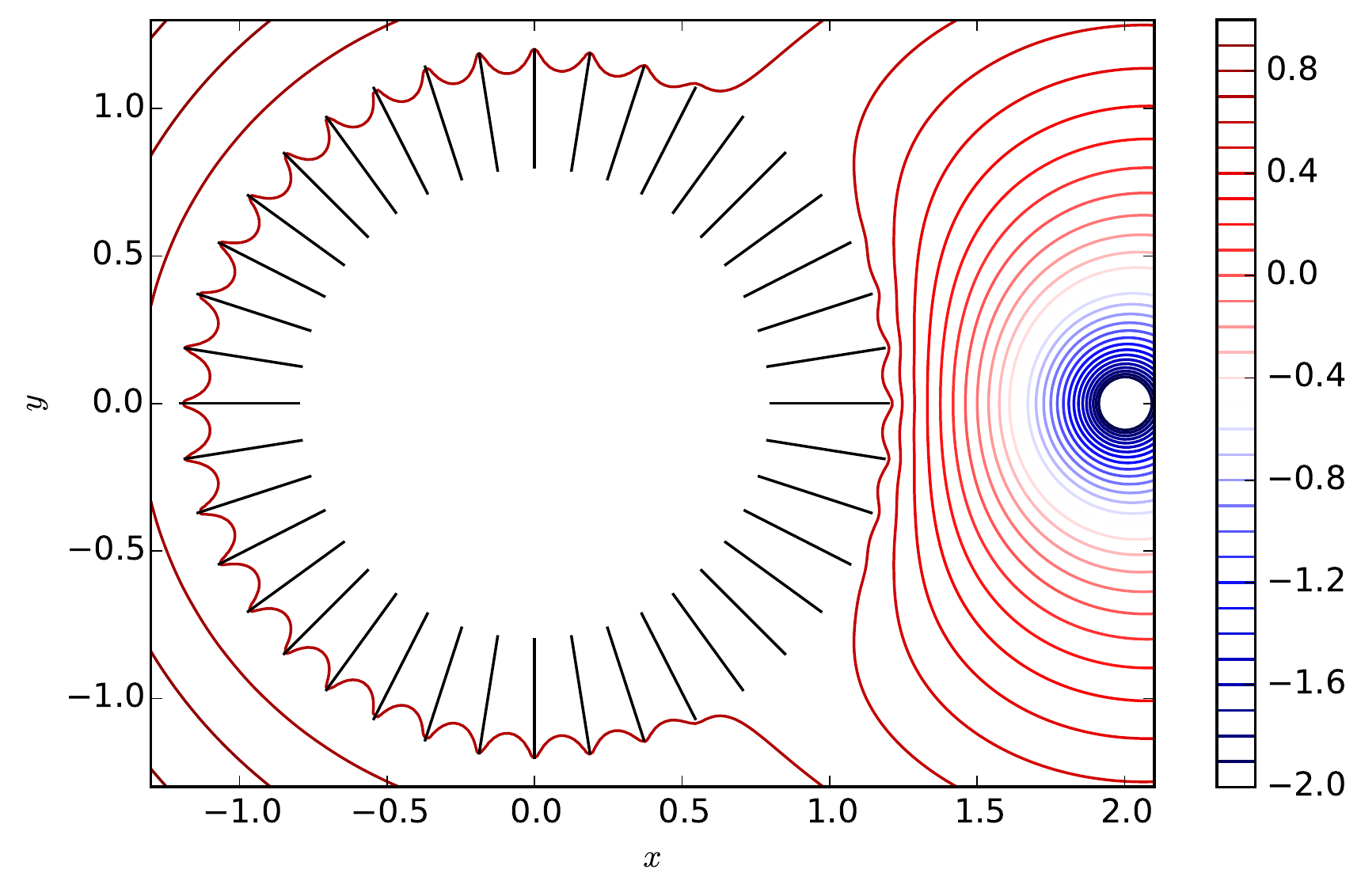}\\
\end{tabular}
\caption{Left: a plot of the solution $u({\bf x})$ with $10$ normal plates with radial parameter $r=10^{-1}$. Right: a plot of the solution $u({\bf x})$ with $40$ tangential plates with the same radial parameter, surpassing the material limit in the original numerical experiments~\cite{Chapman-Hewett-Trefethen-57-398-15}. In both contour plots, $31$ contours are linearly spaced between $-2$ and $+1$.}
\label{fig:Laplace1}
\end{center}
\end{figure}

\begin{figure}[htbp]
\begin{center}
\begin{tabular}{cc}
\includegraphics[width=0.5\textwidth]{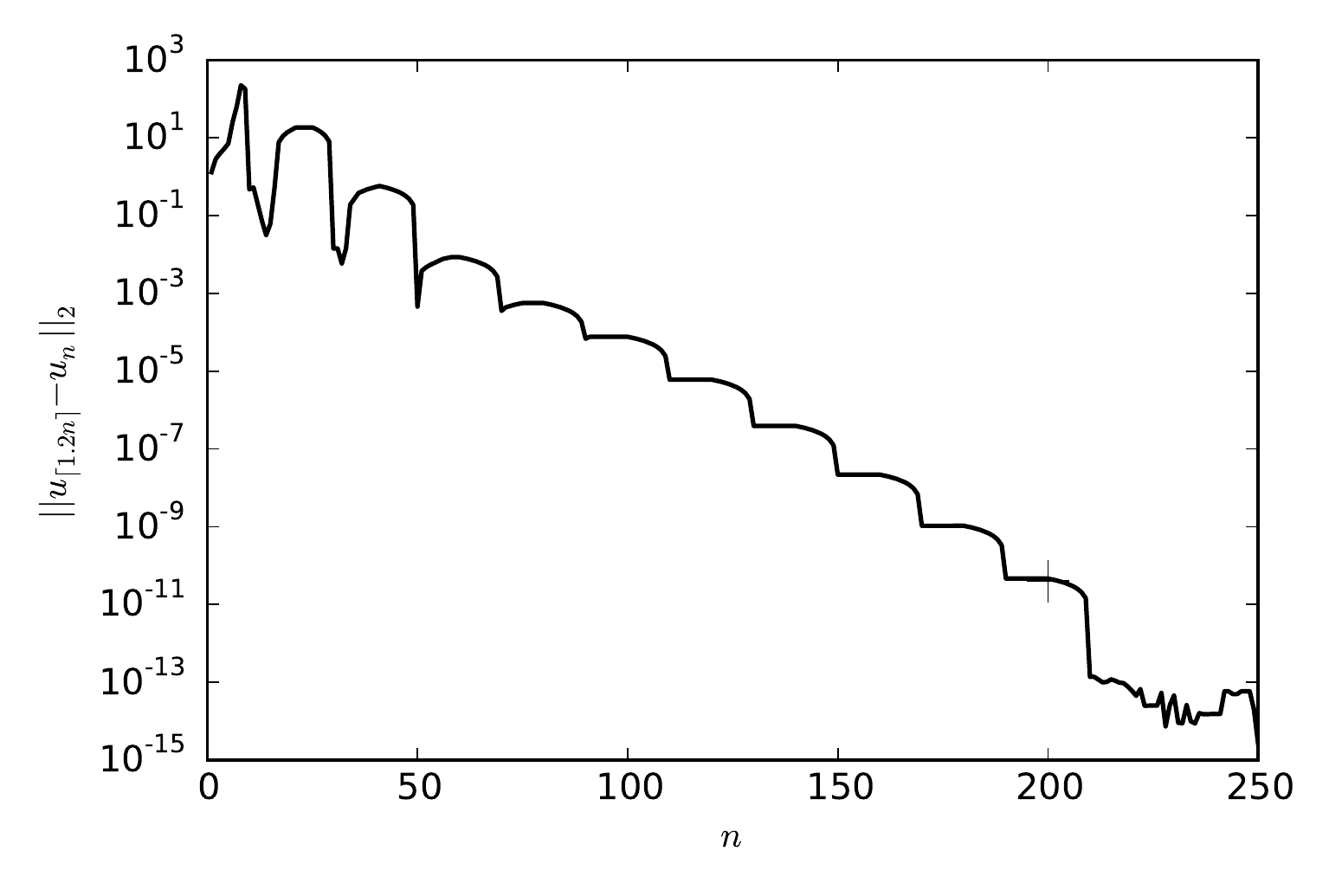}&
\includegraphics[width=0.445\textwidth]{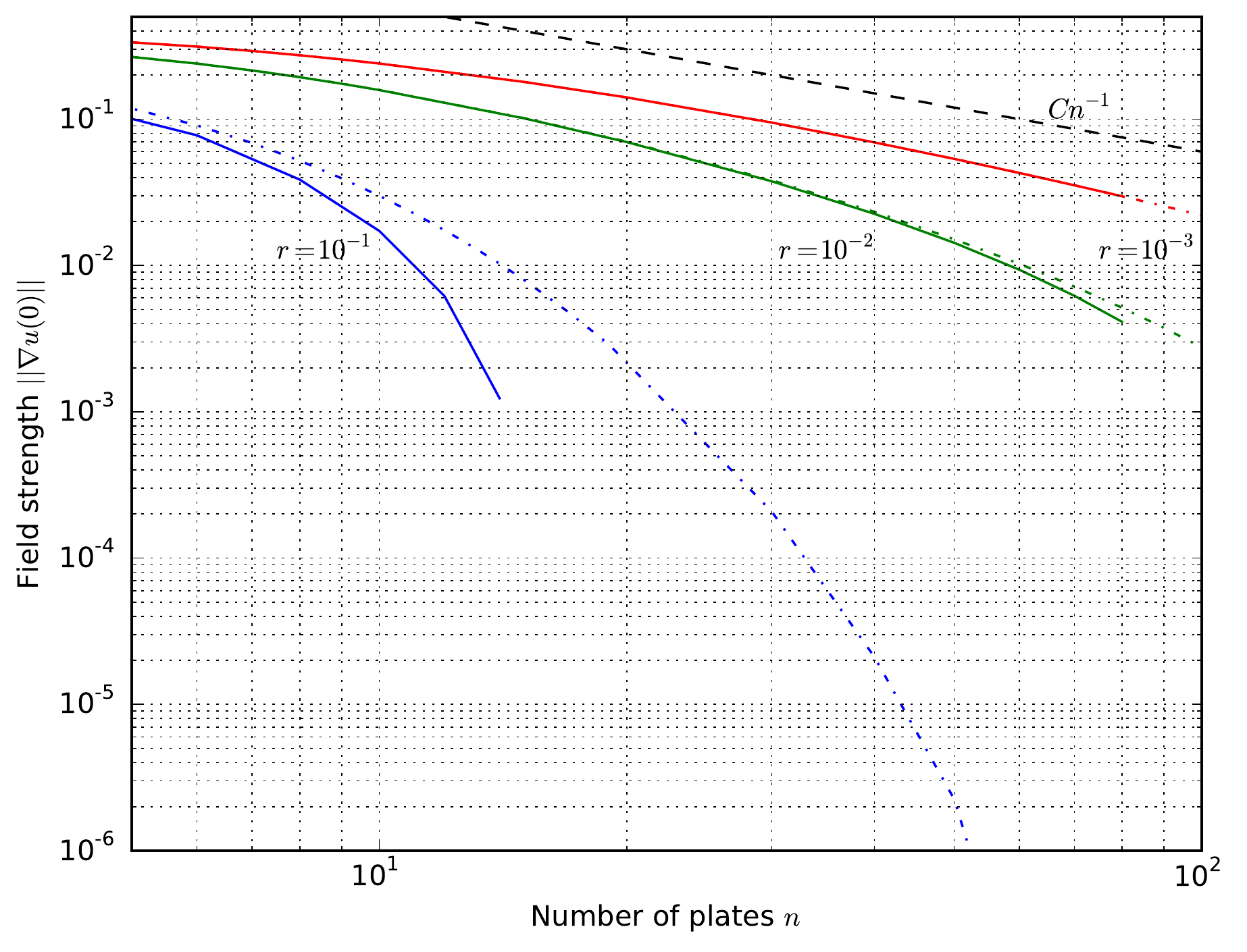}\\
\end{tabular}
\caption{Left: a plot of the Cauchy error of successive approximants for the solution of Laplace's equation with $10$ normal plates with $r=10^{-1}$, corresponding to the left plot in Figure~\ref{fig:Laplace1}, where $n$ is the total number of degrees of freedom. The $+$ indicates where the adaptive QR factorization terminates in double precision, and with this approximation the forward error is $\left\|u_0+{\cal S}_\Gamma[\partial u^s/\partial n]{\rm\,d}\Gamma({\bf y})-u^i\right\|_2 = 2.90\times10^{-15}$ and $\left\|\int_\Gamma[\partial u^s/\partial n]{\rm\,d}\Gamma({\bf y})\right\|_2 = 1.47\times10^{-15}$. Right: a plot of the field strength in the center of the cage versus the number of plates. The dashed lines represent results for normal plates, while the solid lines represent results for tangential plates of the same electrostatic capacity. The normal and tangential plates exhibit different asymptotic scalings.}
\label{fig:Laplace2}
\end{center}
\end{figure}

\subsection{Helmholtz equation with Neumann boundary conditions}\label{subsection:acousticscattering}

The mathematical treatment of the scattering of time-harmonic acoustic waves by infinitely long sound-hard obstacles in three dimensions with simply-connected bounded cross-sections leads to the exterior problem for the Helmholtz equation:
\begin{subequations}
\begin{align}
(\Delta + k^2)u({\bf x}) & = 0, &{\rm for}\quad k\in\mathbb{R},\quad{\bf x}\in \Omega,\\
\dfrac{\partial u({\bf x})}{\partial n({\bf x})} & = 0, &{\rm for}\quad {\bf x}\in\Gamma,\label{eq:NeumannCondition}\\
\lim_{r\to+\infty}\sqrt{r}\left(\dfrac{\partial u^s}{\partial r} - \I k u^s\right) & = 0, &{\rm for}\quad r := |{\bf x}|.\label{eq:SommerfeldCondition}
\end{align}
\end{subequations}
Equation~\eqref{eq:NeumannCondition} enforces sound-hard obstacles, while equation~\eqref{eq:SommerfeldCondition} is the Sommerfeld radiation condition~\cite{Sommerfeld-49}, an explicit radiation condition at infinity. Consider an incident wave with wavenumber $k$ and unit direction ${\bf d}$:
\begin{equation}
u^i({\bf x}) = e^{\I k{\bf d}\cdot {\bf x}}.
\end{equation}
We wish to find the scattered field $u^s$ such that the sum $u = u^i + u^s$ satisfies the Helmholtz equation in the exterior.

The fundamental solution of the Helmholtz equation is proportional to the cylindrical Hankel function of the first kind of order zero~\cite[\S 8.405]{Gradshteyn-Ryzhik-07}:
\begin{equation}
\Phi({\bf x},{\bf y}) = \dfrac{\rm i}{4}H_0^{(1)}(k|{\bf x}-{\bf y}|),
\end{equation}
and the Riemann function is also well known~\cite{Vekua-67} for the Helmholtz equation:
\begin{equation}
\mathfrak{R}(z,\zeta,z_0,\zeta_0) = J_0(k\sqrt{(z-z_0)(\zeta-\zeta_0)}).
\end{equation}
Figure~\ref{fig:Scattering} shows the rank structure of the bivariate kernels and the total solution with a set of randomly generated screens between $[-3,3]$. N.B. it is known that~\cite{Michielssen-Boag-Chew-143-277-96} collinear screens have reduced off-diagonal numerical ranks comparedwith randomly oriented screens.

\begin{figure}[htbp]
\begin{center}
\begin{tabular}{cc}
\includegraphics[width=0.375\textwidth]{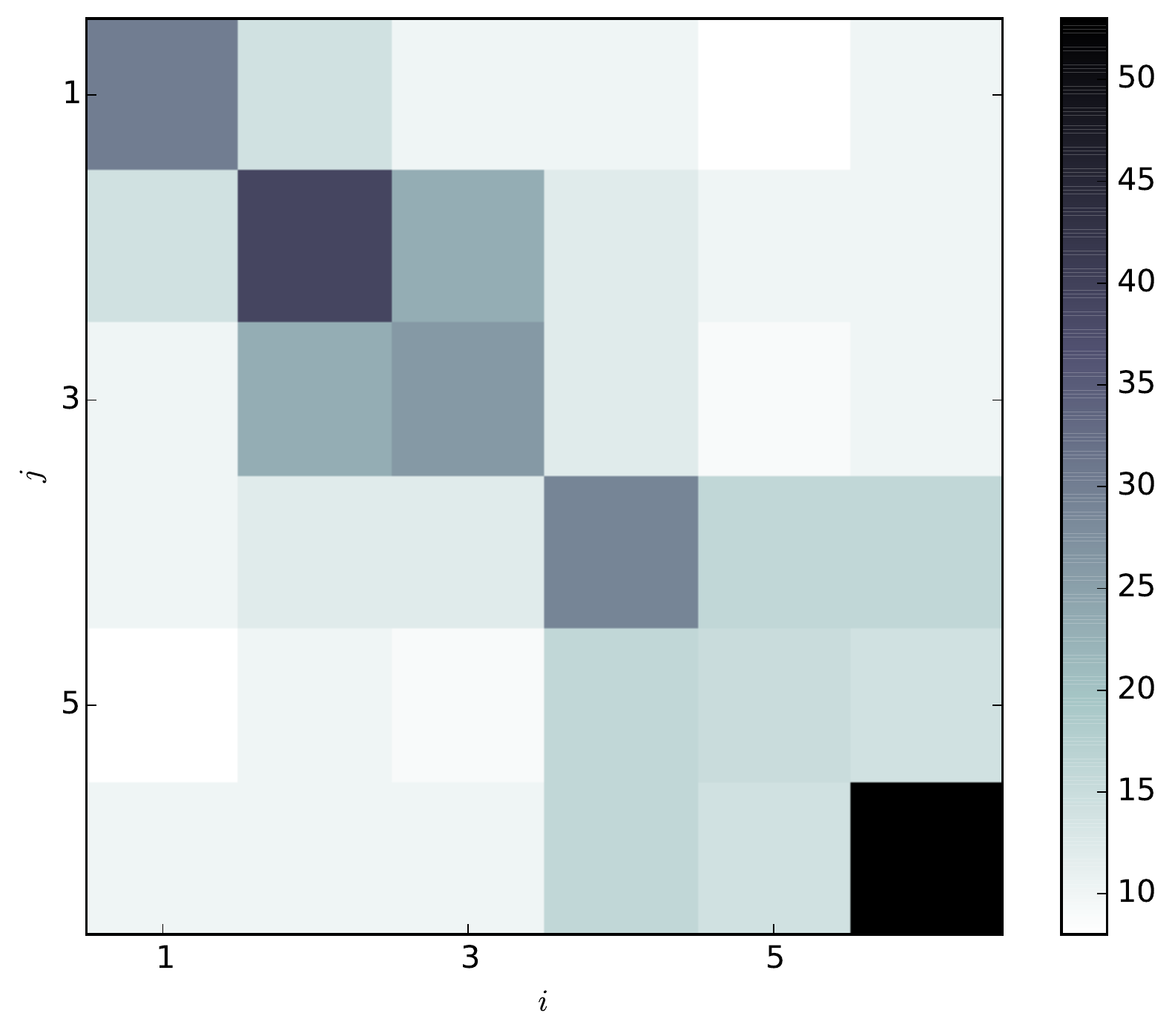}&
\includegraphics[width=0.575\textwidth]{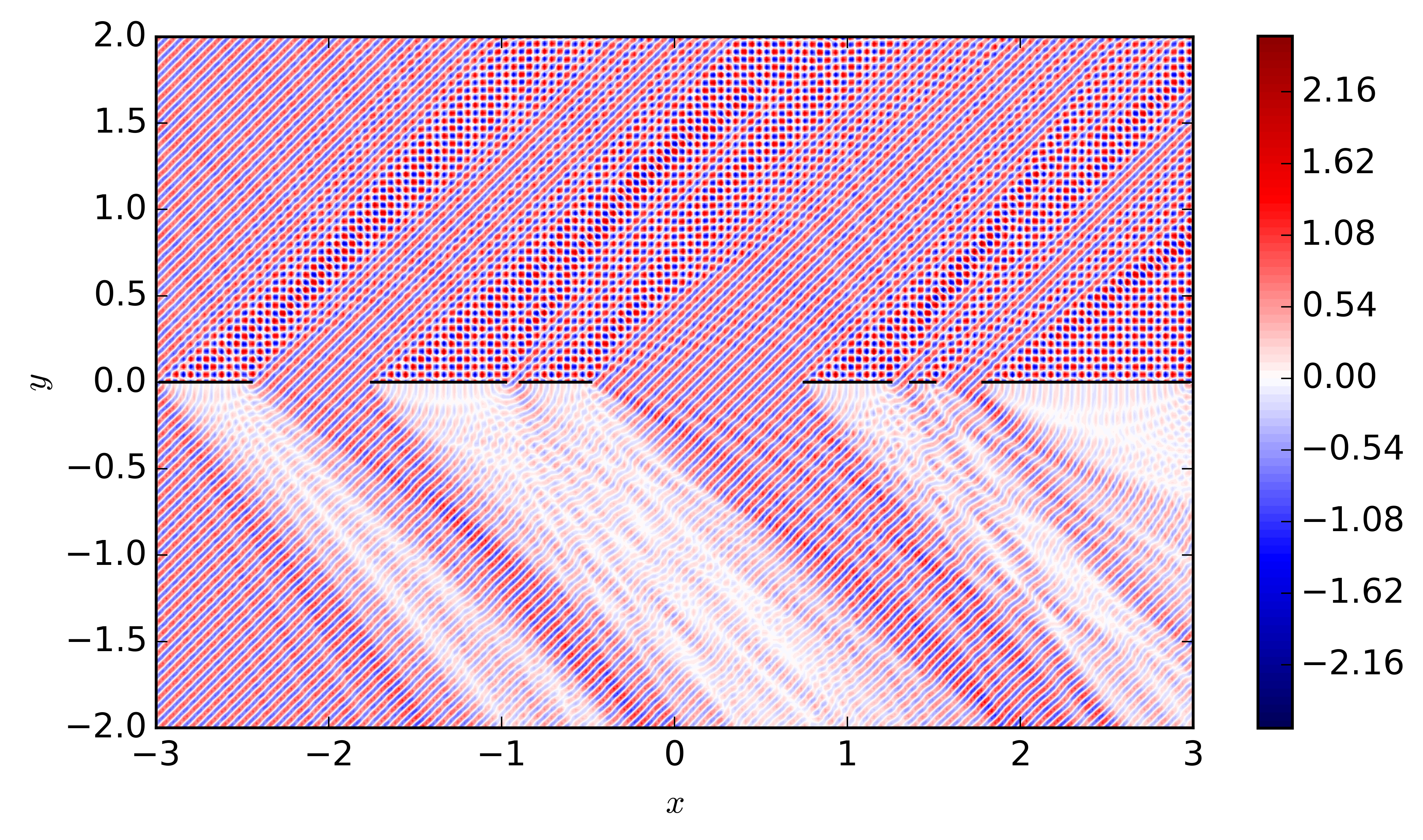}\\
\end{tabular}
\caption{Acoustic scattering with Neumann boundary conditions from an incident wave with $k=100$ and ${\bf d}=(1/\sqrt{2},-1/\sqrt{2})$. Left: a plot of the numerical ranks of $J_0(k|{\bf x}-{\bf y}|)$ connecting domain $i$ to domain $j$, where it can be seen that interaction between domains is relatively weaker than self-interaction. Right: a plot of the total solution. $1,\!392$ degrees of freedom are required to represent the piecewise density in double precision.}
\label{fig:Scattering}
\end{center}
\end{figure}

\subsection{Gravity Helmholtz equation with Dirichlet boundary conditions}

The Helmholtz equation in a linearly stratified medium:
\begin{subequations}
\begin{align}
(\Delta + E + x_2)u({\bf x}) & = 0, &{\rm for}\quad E\in\mathbb{R},\quad{\bf x}\in \Omega,\\
u({\bf x}) & = 0, &{\rm for}\quad {\bf x}\in\Gamma,\label{eq:DirichletCondition}\\
\lim_{x_2\to+\infty}\dfrac{1}{\sqrt{E+x_2}}\int_\mathbb{R}\left|\dfrac{\partial u}{\partial x_2} - \I \sqrt{E+x_2} u\right|^2{\rm\,d}x_1 & = 0, &\label{eq:BarnettConditiona}\\
\lim_{x_2\to-\infty}\int_\mathbb{R}\left|u\right|^2+\left|\dfrac{\partial u}{\partial x_2}\right|^2{\rm\,d}x_1 & = 0, &\label{eq:BarnettConditionb}\\
\lim_{L\to+\infty}\lim_{x_1\to\pm\infty}\int_{-L}^L\left|u\right|^2+\left|\dfrac{\partial u}{\partial x_1}\right|^2{\rm\,d}x_2 & = 0, &\label{eq:BarnettConditionc}
\end{align}
\end{subequations}
models quantum particles of fixed energy in a uniform gravitational field~\cite{Barnett-Nelson-Mahoney-297-407-15}. Equation~\eqref{eq:DirichletCondition} enforces sound-soft obstacles, while equations~\eqref{eq:BarnettConditiona}--\eqref{eq:BarnettConditionc} form an explicit radiation condition at infinity derived in~\cite{Barnett-Nelson-Mahoney-297-407-15}.

The fundamental solution of the Helmholtz equation in a linearly stratified medium is derived in~\cite{Bracher-et-al-66-38-98}:
\begin{equation}
\Phi({\bf x},{\bf y}) = \dfrac{1}{4\pi}\int_0^\infty \exp \I \left[\dfrac{|{\bf x}-{\bf y}|^2}{4t} + \left(E+\dfrac{x_2+y_2}{2}\right)t-\dfrac{1}{12}t^3\right]\dfrac{{\rm d}t}{t}.
\end{equation}
Numerical evaluation via the trapezoidal rule~\cite{Trefethen-Weideman-56-385-14} along a contour of approximate steepest descent on the order of $10^5$ evaluations per second is reported in~\cite{Barnett-Nelson-Mahoney-297-407-15}. This equation is also known as the gravity Helmholtz equation.

Consider an incident fundamental solution with energy $E$ and source ${\bf y}$:
\begin{equation}
u^i({\bf x}) = \Phi({\bf x},{\bf y}).
\end{equation}
We wish to find the scattered field $u^s$ such that the sum $u = u^i + u^s$ satisfies the gravity Helmholtz equation in the exterior. In addition to the fundamental solution, we require the Riemann function of the PDO. With the prospect of deriving a fast numerical evaluation in future work, we prove the following theorem in~\ref{appendix:Riemann}.
\begin{theorem}\label{theorem:GravityHelmholtzRiemann}
The Riemann function of the gravity Helmholtz equation, where $c(x_1,x_2) = E + x_2$ and therefore $C(z,\zeta) = \frac{E}{4}+\frac{z-\zeta}{8\I}$ has the power series:
\begin{equation}
\mathfrak{R}(z,\zeta,z_0,\zeta_0) = 1+ \sum_{i=1}^\infty\sum_{j=1}^\infty A_{i,j}(z-z_0)^i(\zeta-\zeta_0)^j,
\end{equation}
where the coefficients $A_{i,j}$ satisfy~\eqref{eq:RiemannIC}--\eqref{eq:RiemannRec}, and the integral representation:
\begin{align}
V(u,v) = \dfrac{1}{2\pi\I}\int_{\gamma-\I\infty}^{\gamma+\I\infty} &\dfrac{1}{\sqrt{s^2-u/4\I}} \exp\left\{8\I \tilde{E}\left((s^2-u/4\I)^{1/2}-s\right) \right.\nonumber\\
& \left.+ \frac{8\I}{3}\left(s^3-(s^2-u/4\I)^{3/2}\right) + (v-u)s\right\}{\rm\,d}s,\label{eq:GravityHelmholtzRiemannV}
\end{align}
where $\mathfrak{R}(z,\zeta,z_0,\zeta_0) = V(z-z_0,\zeta-\zeta_0)$ and where $\tilde{E} = \dfrac{E}{4} + \dfrac{z_0-\zeta_0}{8\I}$.
\end{theorem}

Figure~\ref{fig:GravityHelmholtz} shows the total solution to the gravity Helmholtz equation with Dirichlet boundary conditions and the $2$-norm condition number of the truncated and preconditioned system.

\begin{figure}[htbp]
\begin{center}
\begin{tabular}{cc}
\includegraphics[width=0.44\textwidth]{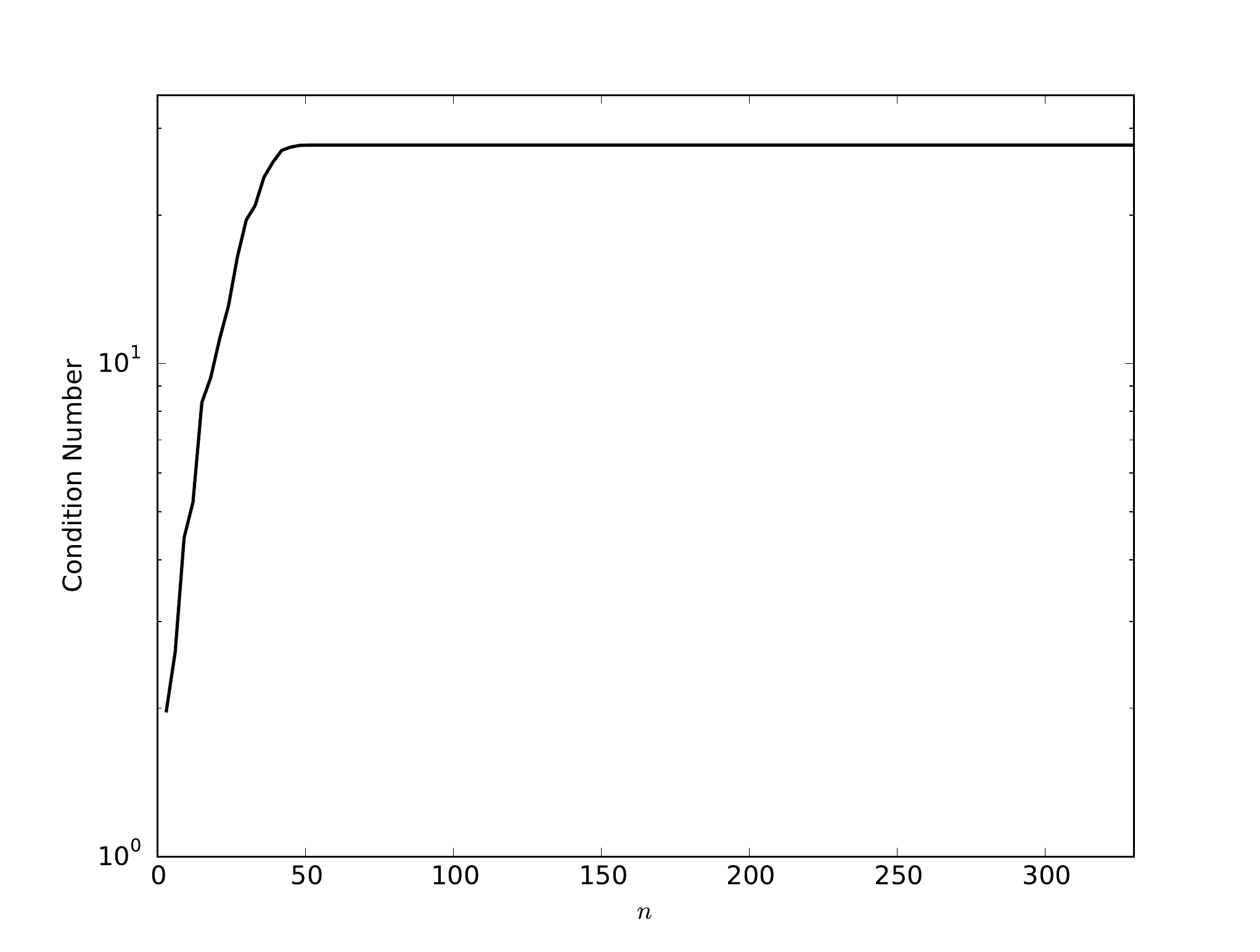}&
\includegraphics[width=0.53\textwidth]{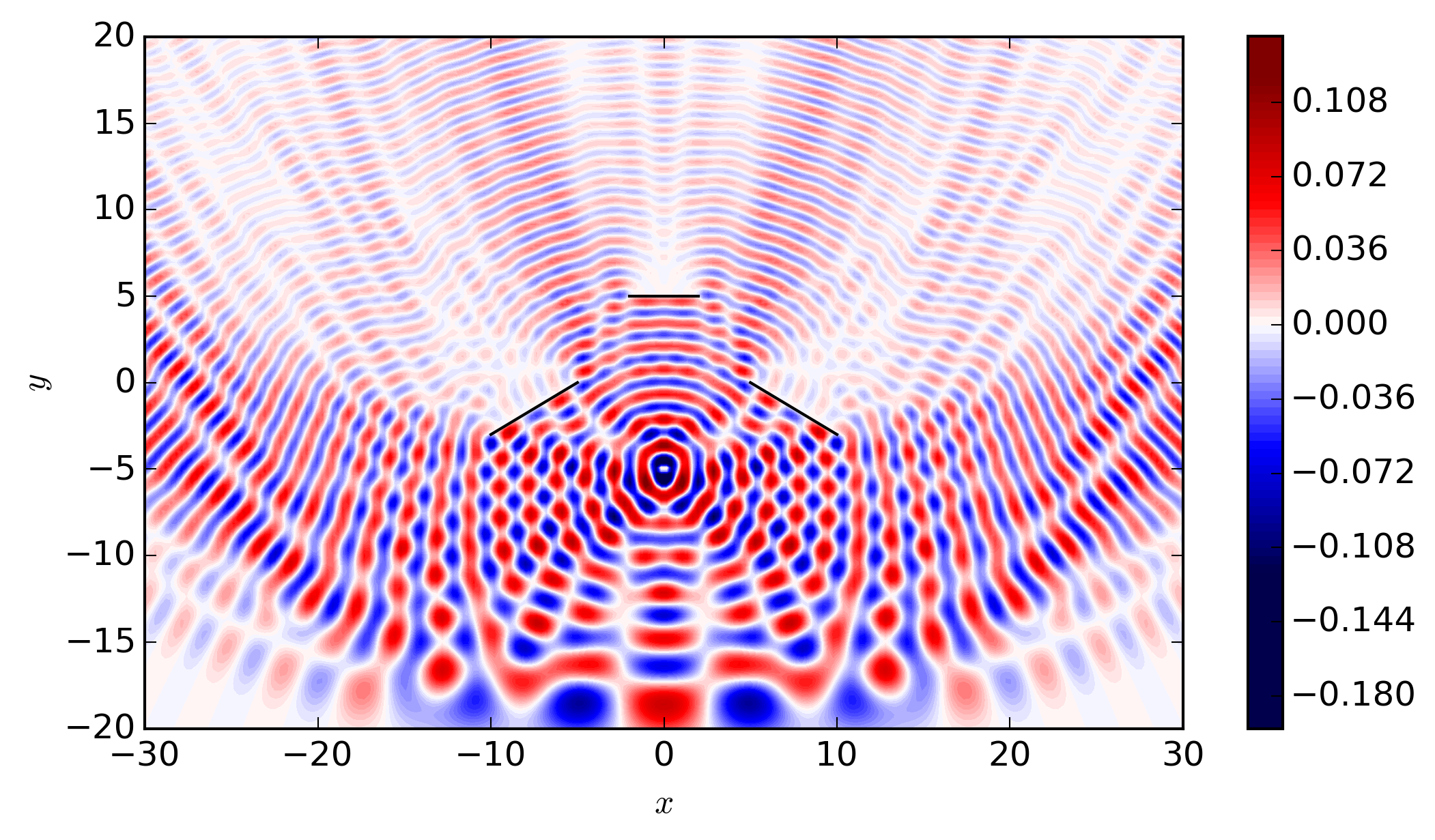}\\
\end{tabular}
\caption{Acoustic scattering with Dirichlet boundary conditions from an incident fundamental solution $\Phi({\bf x},{\bf y})$ with $y=(0,-5)$ and $E=20$ against the sound-soft intervals $((-10,-3),(-5,0))\cup((-2,5),(2,5))\cup((5,0),(10,-3))$. Left: a plot of the $2$-norm condition number of the truncated and preconditioned system with $n$ degrees of freedom. Right: a plot of the total solution. $332$ degrees of freedom are required to represent the piecewise density in double precision.}
\label{fig:GravityHelmholtz}
\end{center}
\end{figure}

\subsection{Helmholtz equation with nearly singular Dirichlet boundary data}

In this application, we consider the Helmholtz equation with nearly singular Dirichlet boundary data.    Consider the scattering of a collection of point sources arbitrarily close to a sound-soft obstacle.   If we parameterize the locations of the point sources by the family of Bernstein ellipses $E_\rho$, then we know that the Chebyshev series representation of the incident wave will have degree which scales as $n = {\cal O}((\log\rho)^{-1})$ as $\rho\to1$. Additionally, as the point sources approach the boundary, the integral operator has bandwidth ${\cal O}(k)$, independent of the Bernstein ellipse parameter. This is a challenging scenario for conventional integral equation solvers since a piecewise polynomial approximation to the nearly singular boundary data may not be much more efficient than a global representation. Furthermore, if the point sources are allowed to move freely on the Bernstein ellipse, then no adaptivity may be used to uniformly accelerate the solvers.  The demonstrations in this section are also applicable to the important problem of many micro swimmers in Stokes flow approaching an obstacle, as the swimmers can be modelled as point sources, see~\cite{Davis-Crowdy-66-53-12}.

This set of problems completely demonstrates the scaling ${\cal O}((m_x+m_y)^2n)$ of our algorithm: the bandwidth scales with the wavenumber, and the degree scales with the reciprocal of the log of the Bernstein ellipse parameter. Additionally, a partial $QR$ factorization of the singular integral operator may be cached or precomputed\footnote{The cached QR factorization can be adaptively grown without   re-computing from scratch by exploiting the fact that  the operator is banded below, thus the number of degrees of freedom ($n$) needed to resolve the solution within a prescribed tolerance need not be known apriori.  This automatic caching of the QR factorization is implemented in {\tt ApproxFun.jl}.}, resulting in the reduced ${\cal O}((m_x+m_y)n)$ complexity for additional solves.    Figure~\ref{fig:DirichletTiming} shows the scalings of the computation for three wavenumbers and varying Bernstein ellipse parameters. The figure also shows the solution of the Helmholtz equation with $100$ nearby source terms.

\begin{figure}[htbp]
\begin{center}
\begin{tabular}{cc}
\includegraphics[width=0.44\textwidth]{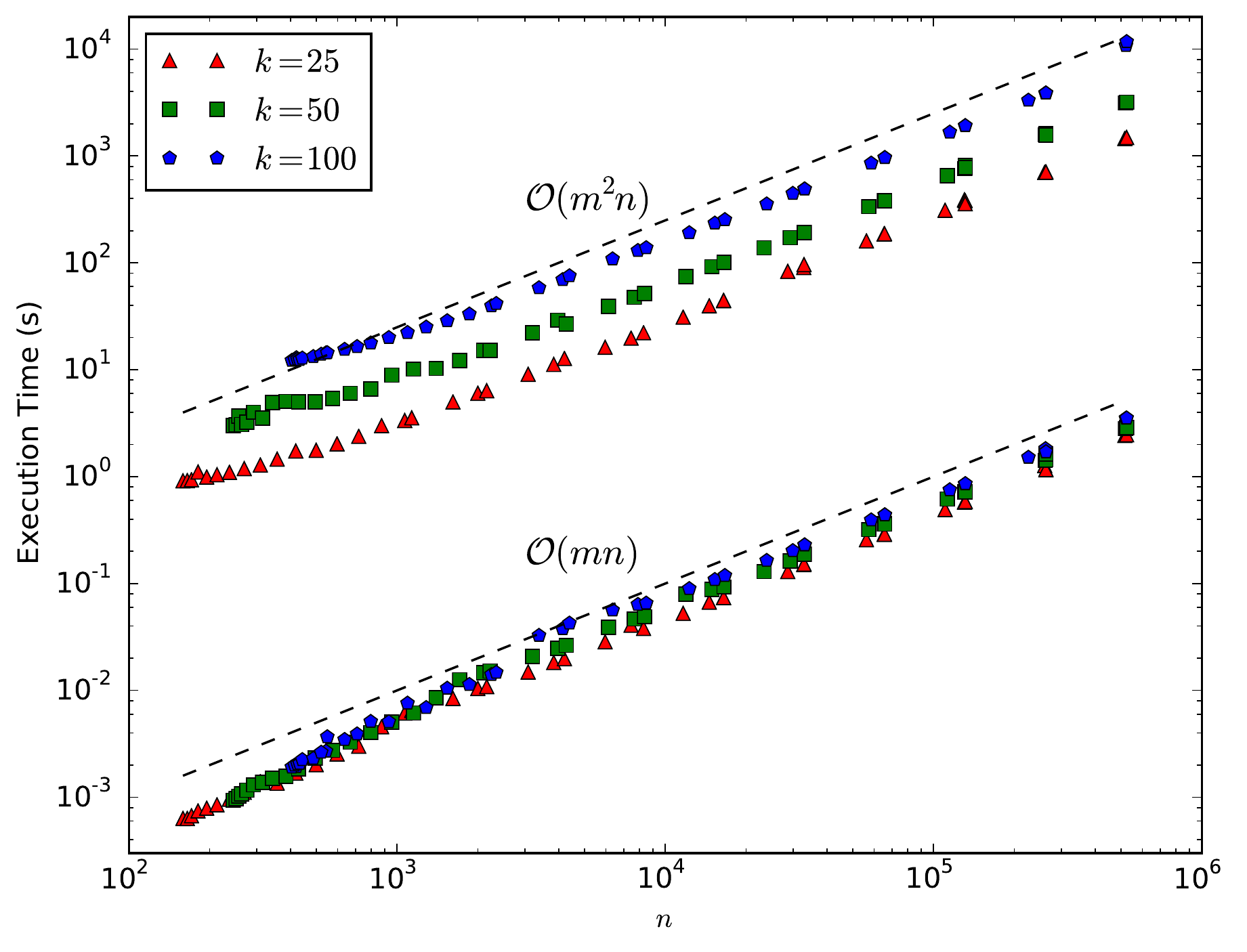}&
\includegraphics[width=0.53\textwidth]{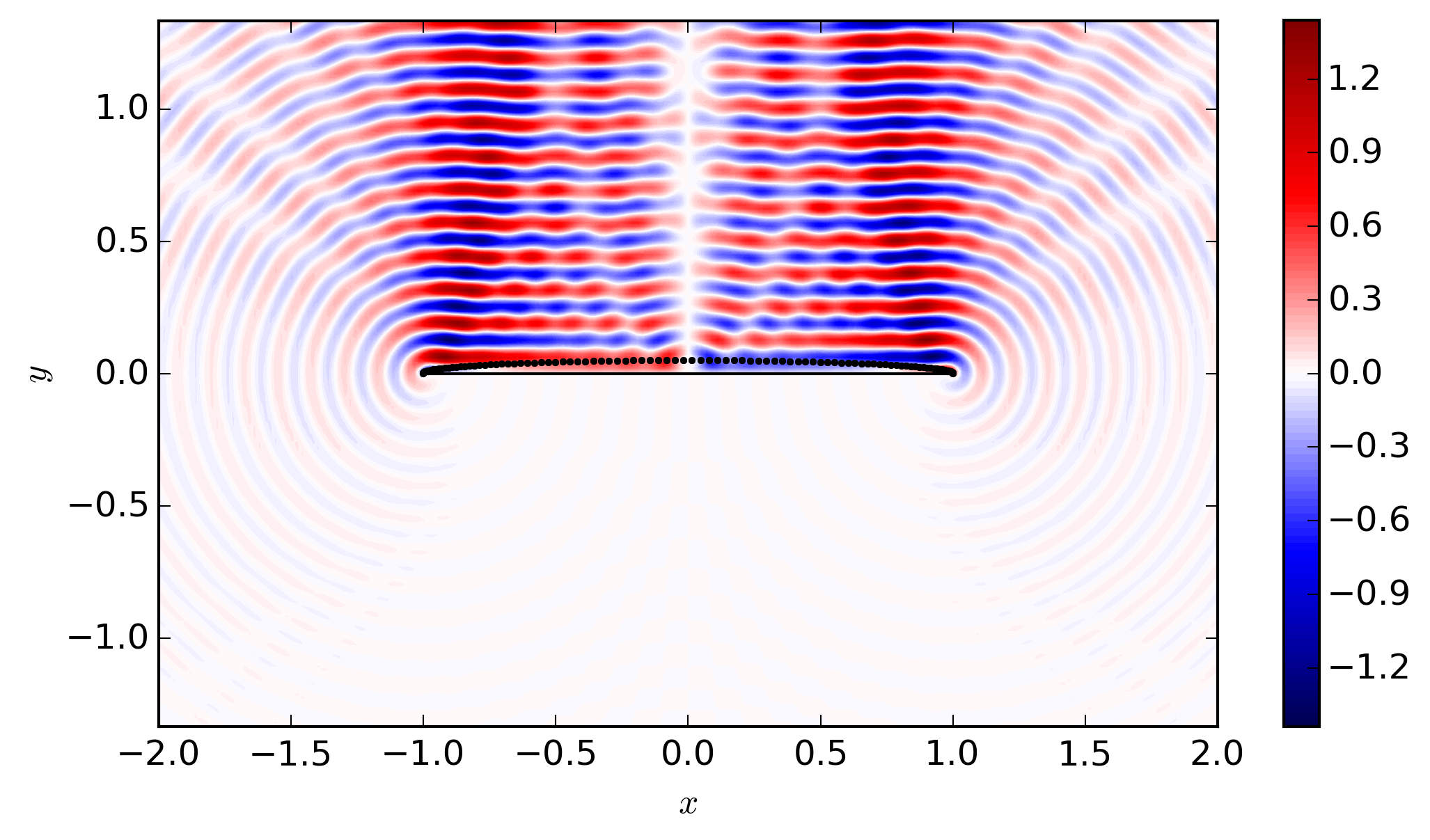}\\
\end{tabular}
\caption{Timings to solve the Helmholtz equation (left) with nearly singular boundary data, and a sample of the solution (right). Left: timings are illustrated for three wavenumbers and multiple Bernstein ellipse parameters resulting in Chebyshev expansions reaching degrees on the order of half a million. The high data set show the partial $QR$ factorization and the back substitution, while the low data set show the reduced time with a cached $QR$ factorization. Right: a sample solution at $k=50$, with $100$ point sources uniformly distributed in angle on the top half of the Bernstein ellipse $E_{1.05}$ with charges $+1$ on the right half and $-1$ on the left half, resulting in the symmetric output.}
\label{fig:DirichletTiming}
\end{center}
\end{figure}

\section{Numerical Discussion \& Outlook}

The software package {\tt SingularIntegralEquations.jl}~\cite{SingularIntegralEquations} written in the {\sc Julia} programming language~\cite{Julia-12,Julia-14} implements the banded singular integral operators, methods relating to bivariate function approximation and construction with diagonal singularities, fast \& spectrally accurate numerical evaluation of scattered fields and several examples including those described in this work. Built on top of {\tt ApproxFun.jl}, {\tt SingularIntegralEquations.jl} uses the adaptive QR factorization described in~\cite{Olver-Townsend-55-462-13} and acts as an extension to the framework for infinite-dimensional linear algebra. All numerical simulations are performed on a MacBook Pro with a $2.8$ GHz Intel Core i7-4980HQ processor and $16$ GB of RAM. While timings are continuously being improved, Table~\ref{table:timings} shows the current timings to solve the problems in section~\ref{section:applications}. All the numerical problems relating to our applications have been abstracted so that to explore a new elliptic PDE in {\tt SingularIntegralEquations.jl}, the user only needs a fast evaluation of the fundamental solution and its Riemann function.

\begin{table}[htbp]
\begin{center}
\caption{Calculation times in seconds to solve the problems in section~\ref{section:applications}. Evaluation of the scattered field is reported per target. Timings for the Laplace equation are for $10$ normal plates.}
\label{table:timings}
\begin{tabular}{rr@{.}lr@{.}lr@{.}l}
\sphline
& \multicolumn{2}{c}{Kernel assembly} & \multicolumn{2}{c}{Adaptive QR} & \multicolumn{2}{c}{Evaluation of scattered field}\\
\sphline
Laplace & 0&888 & 0&518 & 0&0000135\\
Helmholtz ($k=100$) & 1&73 & 67&6 & 0&00652\\
Gravity Helmholtz ($E=20$) & 3&11 & 1&20 & 0&0139\\
\sphline
\end{tabular}
\end{center}
\end{table}%
For problems involving a union of a considerably large number of domains, the current method of interlacing all operators can be improved. In future work on fractal screens motivated by~\cite{Chandler-Wilde-Hewett-14}, alternative algorithms based on hierarchical block diagonalization via a symmetrized Schur complement~\cite{Aminfar-Ambikasaran-Darve-14} may be explored specifically exploiting the low rank off-diagonal structure arising from coercive singular integral operators of elliptic PDOs.  This is close in spirit to the Fast Multipole Method \cite{Greengard-Rokhlin-73-325-87}, but applied to the banded representation of the singular integral operators, instead of discretizations arising from quadrature rules. A preliminary result in this direction is shown in the left side of Figure~\ref{fig:GreatBritishMetacage}.

\begin{figure}[htbp]
\begin{center}
\begin{tabular}{cc}
\includegraphics[width=0.48\textwidth]{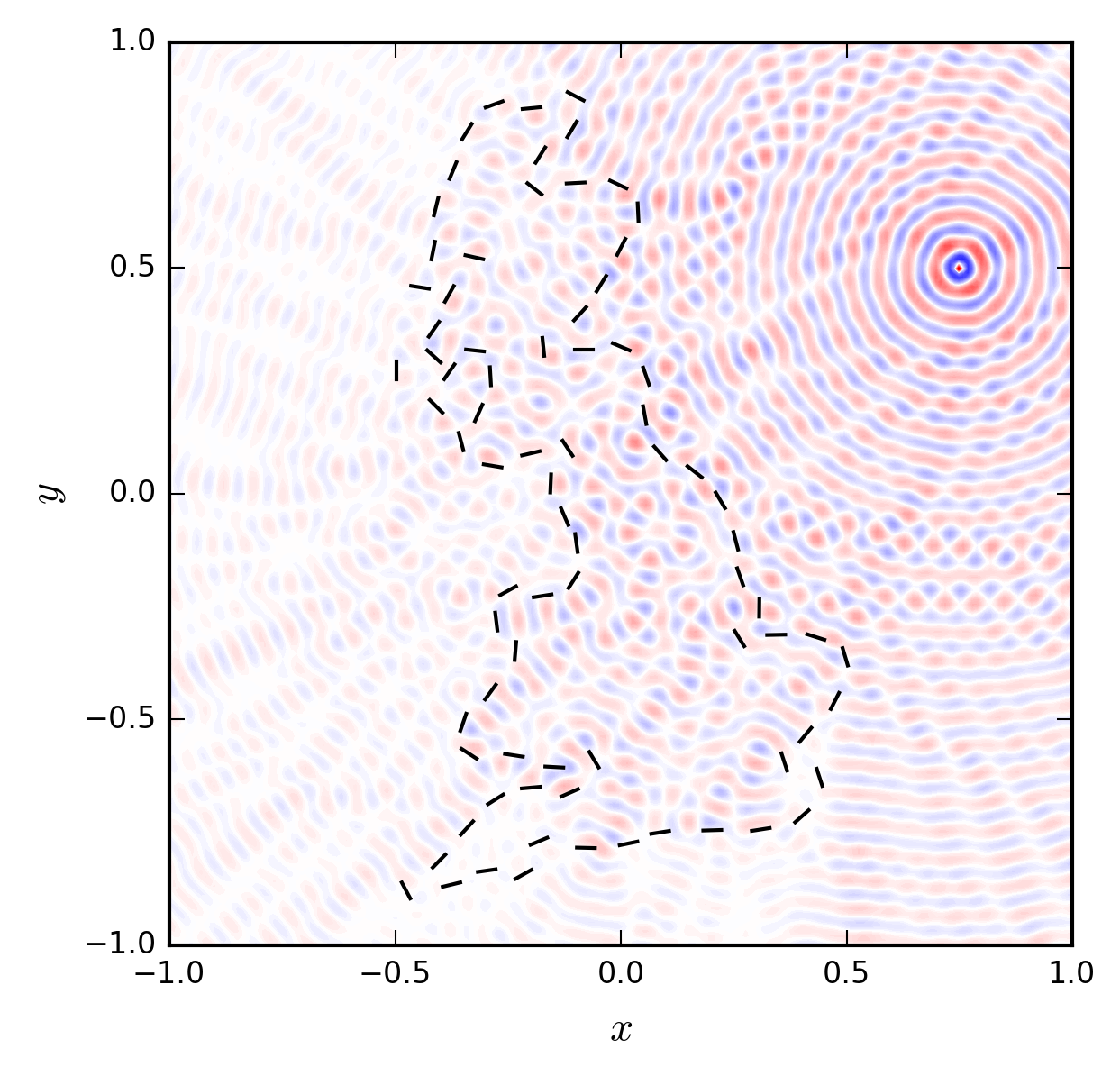}&
\raisebox{.25\height}{\includegraphics[width=0.48\textwidth]{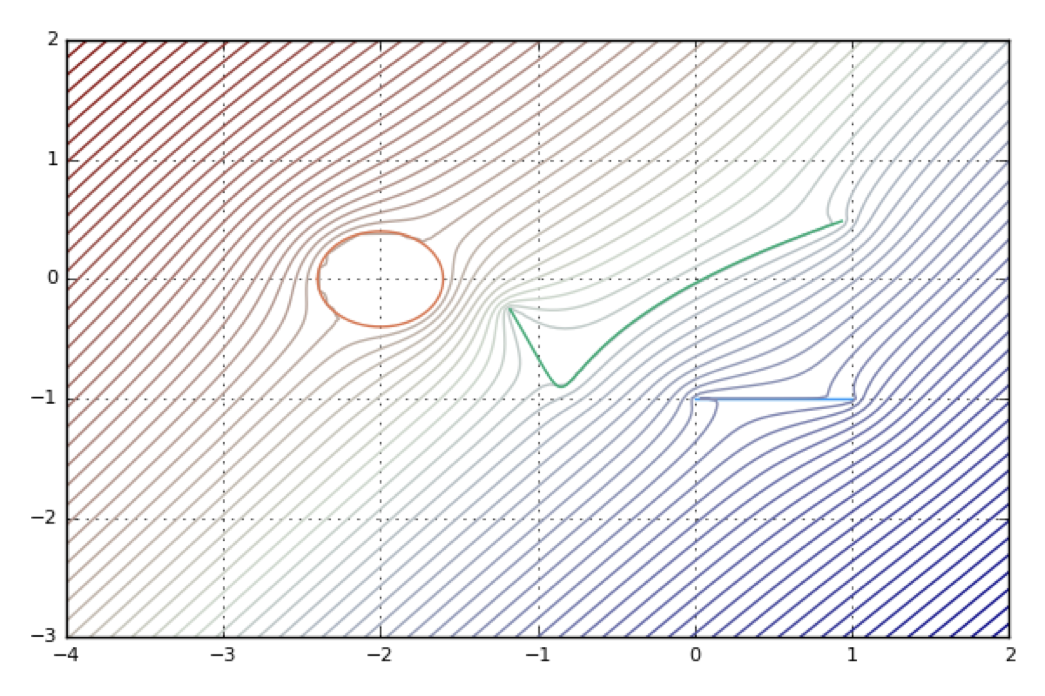} }\\
\end{tabular}
\caption{Left: An illustration of a scenario that benefits from the abstraction of hierarchical matrix factorizations to our hierarchical operators. The Dirichlet solution of the Helmholtz equation with $k=100$ and one incident source in the North Sea. Right: Idealized fluid flow around three obstacles.}
\label{fig:GreatBritishMetacage}
\end{center}
\end{figure}

As illustrated in subsection~\ref{subsection:acousticscattering} on the acoustic scattering of the Helmholtz equation with Neumann boundary conditions, {\tt SingularIntegralEquations.jl} supports higher order diagonal singularities. Future work may explore the feasibility of combining automatic differentiation and differentiation of Chebyshev interpolants to automate the construction of the operators with higher order singularities such that the user need only enter the fundamental solution with its logarithmic splitting described by~\eqref{eq:GFisAlogplusB}.


The approach developed in this article is also adaptable to other domains such as disjoint unions of circles and polynomial maps of intervals and circles, see the right side of Figure~\ref{fig:GreatBritishMetacage} for an example calculated using {\tt SingularIntegralEquations.jl} of idealized fluid flow over three  domains: an interval, a circle and a polynomial map of an interval.  To take into account circles, a similar analysis is straightforward with Laurent polynomials in place of weighted Chebyshev polynomials. However, a combined field formulation is beneficial to ensure well-conditioning when the solution of the exterior problem is near an eigenmode of the interior problem. Equations over  maps of the unit interval and circle   can also be reduced to numerically banded singular integral operators  via approximating the map by a polynomial and using the spectral mapping theorem.  The key formula in the Hilbert case is derived in \cite[Theorem 5.32]{Trogdon-Olver-15}, which implies that the   Hilbert transform over a polynomial map of the unit interval can be reduced to a compact perturbation of the Hilbert transform over the unit interval.  Expanding on this result, as well as adapting the procedure to log transforms, will be the topic of a subsequent publication.   Future work may consider the use of these modified Chebyshev series for banded operators when two disjoint contours are in close proximity.   When two or more contours coalesce, banded singular integral operators will depend on the ability to produce the orthogonal polynomials associated with that domain.  Densities of the single- and double-layer potentials will have singularities on domains with cusps. Such an analysis is undetermined.

As discussed in~\cite{Barnett-Nelson-Mahoney-297-407-15}, the fundamental solution of the gravity Helmholtz equation has an analogy to the Schr\"odinger equation with a linear potential. The Helmholtz equation with a parabolic refractive index shares the same analogy and the fundamental solution is also known~\cite{Constantinou-Thesis-91,Heller-91}. Parabolic refractive indices occur when considering the shielding of optical fibres, leading to Gaussian beams. Scattering problems in this context may shed light on the effects when optical fibres are occluded. Fast and accurate numerical evaluation of the fundamental solution as well as the Riemann function may also be possible via the trapezoidal rule.

An important area of future research is extending the method to higher dimensional singular integral equations. The ultraspherical spectral method was extended to automatically solve general linear partial differential equations on rectangles~\cite{Townsend-Olver-299-106-15} and the ideas used to do this successfully may well translate to singular integral equations.

\section*{Acknowledgments}

We wish to thank Jared Aurentz, Folkmar Bornemann, Dave Hewett, Alex Townsend and  Nick Trefethen for stimulating discussions related to this work. We acknowledge the generous support of the Natural Sciences and Engineering Research Council of Canada (RMS) and the Australian Research Council (SO).

\bibliography{/Users/Mikael/Bibliography/Mik}

\appendix

\section{Proof of Theorem~\ref{theorem:GravityHelmholtzRiemann}}~\label{appendix:Riemann}
To immediately satisfy the boundary conditions~\eqref{eq:Riemannbcs}, we start with the ansatz:
\begin{equation}
\mathfrak{R}(z,\zeta,z_0,\zeta_0) = 1+ \sum_{i=1}^\infty\sum_{j=1}^\infty A_{i,j}(z-z_0)^i(\zeta-\zeta_0)^j,
\end{equation}
and we insert it into the integral equation~\eqref{eq:RiemannIE}:
\begin{align}
&\sum_{i=1}^\infty\sum_{j=1}^\infty A_{i,j}(z-z_0)^i(\zeta-\zeta_0)^j \nonumber\\
& + \int_{z_0}^z\int_{\zeta_0}^\zeta \left(\frac{E}{4}+\frac{t-\tau}{8\I}\right)\left(1+\sum_{i=1}^\infty\sum_{j=1}^\infty A_{i,j}(t-z_0)^i(\tau-\zeta_0)^j\right){\rm\,d}\tau{\rm\,d}t = 0.
\end{align}
With the initial values:
\begin{equation}\label{eq:RiemannIC}
A_{1,1} = -\frac{E}{4}-\frac{(z_0-\zeta_0)}{8\I},\qquad A_{2,1} = -\dfrac{1}{16\I},\qquad A_{1,2} = \dfrac{1}{16\I}, \qquad A_{2,2} = A_{1,1}^2/4,
\end{equation}
and the additional values:
\begin{equation}
A_{i,1} = A_{1,i} = 0,\quad{\rm for}\quad i>2,\qquad A_{i,j} = 0,\quad{\rm for}\quad i\le0,j\le0.
\end{equation}
the coefficients are found to satisfy in general:
\begin{equation}\label{eq:RiemannRec}
ij A_{i,j} + \left(\frac{E}{4}+\frac{z_0-\zeta_0}{8\I}\right)A_{i-1,j-1} - \frac{1}{8\I}A_{i-1,j-2} + \frac{1}{8\I}A_{i-2,j-1} = 0.
\end{equation}
The growth in the constant in front of $A_{i,j}$  ensures that coefficients decay at least exponentially fast, hence the power series converges for all $z$ and $\zeta$.

To get an integral representation for the Riemann function, we start from the differential equation it satisfies after the change of variables $u=z-z_0$ and $v = \zeta-\zeta_0$:
\begin{equation}
\dfrac{\partial^2V}{\partial u\partial v} + \left(\tilde{E}+\frac{u-v}{8\I}\right)V = 0,\qquad \tilde{E} = \dfrac{E}{4}+\dfrac{z_0-\zeta_0}{8\I},
\end{equation}
together with $V(0,v) = V(u,0) = 1$.

Taking the Laplace transform:
\begin{equation}
{\cal L}\{f\}(s) = \int_0^\infty f(v)e^{-sv}{\rm\,d}v,
\end{equation}
of the differential equation, we obtain:
\begin{equation}
s\dfrac{\partial \hat{V}}{\partial u} + \dfrac{1}{8\I}\dfrac{\partial\hat{V}}{\partial s} + \left(\tilde{E} + \dfrac{u}{8\I}\right)\hat{V} = 0, \qquad \hat{V}(0,s) = \dfrac{1}{s}.
\end{equation}
Using the method of characteristics for this first-order PDE, we obtain the general solution as:
\begin{equation}
\hat{V}(u,s) = \exp\left(-8\I \tilde{E}s + \frac{8\I s^3}{3} - us + f(4\I s^2-u)\right).
\end{equation}
The particular solution satisfying the initial condition is:
\begin{align}
\hat{V}(u,s) = \dfrac{1}{\sqrt{s^2-u/4\I}}&\exp\left\{8\I \tilde{E}\left((s^2-u/4\I)^{1/2}-s\right) \right.\nonumber\\
& \left.+ \frac{8\I}{3}\left(s^3-(s^2-u/4\I)^{3/2}\right) - us\right\}.
\end{align}
Inverting the Laplace transform using the Bromwich integral, we find the solution~\eqref{eq:GravityHelmholtzRiemannV}.

\end{document}